%% file: dpglectures.tex
\documentclass[11pt,reqno]{amsart}
\usepackage{amsmath}
\usepackage{amssymb}
\usepackage{amscd}
\usepackage{centernot} 
\usepackage{enumerate}
\usepackage{mathabx}
\usepackage{eucal}
\usepackage{epic}
\usepackage{graphicx}
\usepackage{wrapfig}
\usepackage[all]{xy}
\usepackage{hyperref} 

\theoremstyle{plain}
\newtheorem{theorem}{Theorem}
\newtheorem{lemma}[theorem]{Lemma}
\newtheorem{prop}[theorem]{Proposition}

\theoremstyle{remark}
\newtheorem{remark}[theorem]{Remark}       

\newtheorem{example}[theorem]{Example}
\newtheorem{exercise}[theorem]{Exercise}
\newtheorem{assumption}[theorem]{Assumption}

\newtheorem*{definition*}{Definition}
\newtheorem{definition}[theorem]{Definition}

\input{preamble}

\renewcommand{\grad}{\mathop{\mathrm{grad}}}
\renewcommand{\Hdiv}[1]{H(\dive,{#1})}

\newcommand{\fenics}{{\tt{FEniCS}}}
\newcommand{\trn}{\mathop{\mathrm{tr}_n}}

\newcommand{\pirt}{\mathop{\varPi_{\scriptscriptstyle{\mathrm{RT}}}}}

\newcommand{\hu}{\hat u}
\newcommand{\hb}{\hat b}
\newcommand{\hx}{\hat x}
\newcommand{\hX}{\hat X}
\newcommand{\hq}{\hat q}
\newcommand{\hqh}{\hat q_{n,h}}
\newcommand{\hr}{\hat r}

\newcommand{\hs}{\hat s}
\newcommand{\hsh}{\hat s_{n,h}}

\newcommand{\jmpn}[1]{\big| [#1] \big|_{\d\oh}}

\newcommand{\om}{\varOmega}
\newcommand{\oh}{{\varOmega_h}}
\DeclareMathOperator{\aarg}{arg}
\newcommand{\Xoh}{X_{h,0}}
\newcommand{\Yhopt}{Y_h^\mathrm{opt}}
\newcommand{\Yohopt}{Y_{h,0}^\mathrm{opt}}
\newcommand{\Yohr}{Y_{h,0}^r}
\newcommand{\Yo}{Y_{0}}

\newcommand{\xh}{x_h}

\newcommand{\Xh}{X_h}

\newcommand{\yh}{y_h}

\title{Five lectures on DPG methods}
\author{Jay Gopalakrishnan}
\address{\hrule \vspace{0.1cm} 
  {\it Address:} PO Box 751, Portland State University, Portland, OR 97207-0751.}
\email{gjay@pdx.edu}

\begin{document}

\noindent{\framebox{\parbox{0.99\textwidth}{
\noindent {Lecture Notes}\hfill Jay Gopalakrishnan
\vspace{0.1cm}\hrule \vspace{0.2cm}

\bigskip

\noindent {\bf{FIVE LECTURES ON DPG METHODS}} \hfill
Spring 2013, Portland

\vspace{0.2cm} 
}}}

\bigskip

These lectures present a relatively recent introduction into the class
of discontinuos Galerkin (DG) methods, named {\em Discontinuous
  Petrov-Galerkin} (DPG) methods.  DPG methods, in which DG spaces
form a critical ingredient, can be thought of as least-square methods
in nonstandard norms, or as Petrov-Galerkin methods with special test
spaces, or as a nonstandard mixed method.  We will pursue all these
points of view in this lecture.

By way of preliminaries, let us recall two results from the
\Babuska-Brezzi theory.  Throughout, $b(\cdot,\cdot) : X \times Y \to
\CCC$ is a continuous sesquilinear form where $X$ and $Y$ are
(generally not equal) \nlss\ over the complex field $\CCC$, $Y^*$
denotes the space of continuous conjugate-linear functionals on $Y$,
and $\ell \in Y^*$.

\begin{theorem}
  \label{thm:BNB}
  Suppose $X$ is a Banach space and $Y$ is a reflexive Banach
  space. The following three statements are equivalent:
  \begin{enumerate}[\quad a)\;]
  \item \label{item:bnb1} 
    For any $\ell\in Y^*$, there is a unique $x\in X$ satisfying 
    \begin{equation}
      \label{eq:weakform}
      b(x,y) = \ell(y)\qquad\forall y\in Y. 
    \end{equation}
  \item \label{item:bnb2} $ \{ y\in Y: \; b(z,y)=0, \;\forall z\in X\}
    = \{0\}$ and there is a $C_1>0$ such that
    \begin{gather}
      \label{eq:inf-sup}
      \inf_{0\ne z \in X} \sup_{0\ne y \in Y} 
      \frac{| b(z,y)|}{ \| z \|_X \| y \|_Y } \ge C_1,
    \end{gather}
  \item \label{item:bnb3} $\{ z\in X: \; b(z,y)=0, \;\forall y\in Y\} = \{0\}$
    and there is a $C_1>0$ such that 
    \begin{gather}
      \label{eq:inf-sup-adj}
      \inf_{0\ne y \in Y} \sup_{0\ne z \in X} 
      \frac{| b(z,y)|}{ \| y \|_Y \| z \|_X } \ge C_1.
    \end{gather}
  \end{enumerate}
\end{theorem}

\begin{theorem}
  \label{thm:BNBapprox}
  Suppose $X$ and $Y$ are Hilbert spaces, $X_h \subset X$ and
  $Y_h\subset Y$ are finite dimensional subspaces,
  $\dim(X_h)=\dim(Y_h)$, and suppose one of (\ref{item:bnb1}),
  (\ref{item:bnb2}) or (\ref{item:bnb3}) of Theorem~\ref{thm:BNB}
  hold. If, in addition, there exists a $C_3>0$ such that
  \begin{equation}
    \label{eq:discrete-inf-sup}
    \inf_{0\ne z_h \in X_h} 
    \sup_{0\ne y_h\in Y_h} \frac{ | b(z_h,y_h) |} { \| y_h \|_Y}
    \ge
    C_3,
  \end{equation}
  then there is a unique $x_h\in X_h$ satisfying 
  \begin{equation}
    \label{eq:2}
    b(x_h,y_h) = \ell(y_h)\qquad\forall y_h\in Y_h, 
  \end{equation}
  and 
  \[
  \| x- x_h \|_X \le \frac{C_2}{C_3} \inf_{z_h\in X_h } \| x - z_h \|_X,
  \]
  where $C_2>0$ is any constant for which the inequality $|b(x,y)|
  \le C_2 \| x \|_X \|y\|_Y$ holds for all $x\in X$ and $y\in Y$.
\end{theorem}

We studied these well-known theorems in earlier lectures -- see your earlier
class notes for their full proofs and examples. Methods of the form~\eqref{eq:2}
with $X_h \ne Y_h$ are called {\em Petrov-Galerkin (PG) methods} with
{\em trial space} $X_h$ and {\em test space} $Y_h$.  Note the standard
difficulty: The {\em inf-sup condition}~\eqref{eq:inf-sup} does
{\underline{not}} in general imply the {\em discrete inf-sup
  condition}~\eqref{eq:discrete-inf-sup}.

\section{Optimal test spaces}
\label{sec:optimal-test-spaces}

Although~\eqref{eq:inf-sup}$\centernot\implies$~\eqref{eq:discrete-inf-sup}
in general, we now ask: Is it possible to find a test space~$Y_h$
for which ~\eqref{eq:inf-sup}$ \implies$~\eqref{eq:discrete-inf-sup}?
We now show that the answer is simple and affirmative. From now on,
$X$ and $Y$ are Hilbert spaces, $(\cdot,\cdot)_Y$ denotes the inner
product on $Y$, and $X_h$ is a finite-dimensional subspace of $X$
(where $h$ is some parameter related to the dimension).

\begin{definition}
  Given any trial space $X_h$, we define its {\bf \em optimal test space}
  for the continuous sesquilinear form $b(\cdot,\cdot) : X \times Y
  \to \CCC$ by
\[
\Yhopt = T(X_h)
\]
where $T: X \to Y$ (the {\em trial-to-test operator}) is defined by 
\begin{equation}
  \label{eq:T}
  (T z, y)_Y = b (z, y) \qquad\forall y\in Y, \; z\in X.
\end{equation}
Equation~\eqref{eq:T} uniquely defines a $Tz$ for any given $z\in X$,
by Riesz representation theorem.  We call $Tz$ the ``optimal'' test
function of $z$, because it solves an optimization problem, as we see
next.
\end{definition}

\begin{prop}[Optimizer]
  \label{prop:optimalopti}
  For any $z\in X$, the maximum of  
  \[
  f_z(y)= \frac{ | b(z,y) |}{ \| y \|_Y}
  \]
  over all nonzero $y\in Y$ is attained  at $y=Tz$.
\end{prop}
\begin{proof}
  By duality in Hilbert spaces,
  \begin{align*}
    \sup_{0\ne y\in Y} f_z(y) & = \sup_{0\ne y\in Y} 
    \frac{ |(Tz,y)_Y|}{\| y \|_Y } = \| Tz \|_Y,
  \end{align*}
  and $f_z(Tz) = \|Tz\|_Y.$
\end{proof}

\begin{prop}[Exact inf-sup condition $\implies$ Discrete inf-sup condition] 
  \label{prop:inf-sup-follows}
  If~\eqref{eq:inf-sup} holds, then~\eqref{eq:discrete-inf-sup} holds
  with $C_3=C_1$ when we set $Y_h = \Yhopt$.
\end{prop}
\begin{proof}
  For any $z_h \in X_h$, letting 
  \[
  s_1 = \sup_{0\ne y\in Y} \frac{ | b( z_h, y) |}{ \| y \|_Y},
  \qquad
  s_2 = \sup_{0\ne y_h\in \Yhopt} \frac{ | b( z_h, y_h) |}{ \| y_h \|_Y}
  \]
  it is obvious that $s_1 \ge s_2$. To prove that $s_1 \le s_2$, since
  $s_1 = \| T z_h \|_Y$ by Proposition~\ref{prop:optimalopti},\\
  \[
    s_1
     = \| T z_h \|_Y= \frac{ |(Tz_h,T z_h)_Y|}{ \| Tz_h\|_Y}
    \le \sup_{y_h\in \Yhopt} 
    \frac{ |(Tz_h, y_h)_Y|}{ \| y_h\|_Y}
    = \sup_{y_h\in \Yhopt} 
    \frac{ |b(z_h, y_h)|}{ \| y_h\|_Y} =s_2,
  \]
  so $s_1 = s_2$. Hence the discrete inf-sup constant equals the exact
  inf-sup constant.
\end{proof}

\begin{definition}
  For {\em any} trial subspace $X_h \subset X$, the {\bf{\em ideal PG
      method}} finds $x_h\in {\Xh}$ solving 
  \begin{equation}
    \label{eq:ipg}
      b({\xh},{\yh}) = \ell({\yh}),
      \qquad\forall {\yh} 
      \in \Yhopt.
  \end{equation}
\end{definition}

\begin{assumption}
  \label{asm:wellposed}
  Suppose $\{ z \in X: \;b(z, y)=0,\; \forall y\in Y\} = \{0\}$ and
  suppose there exist $C_1, C_2>0$ such that
  \[
  C_1 \| y \|_Y \le 
  \sup_{0\ne z \in X} 
  \frac{| b(z,y)|}{ \| z \|_X } \le C_2 \| y \|_Y\qquad \forall y\in Y.
  \]
  The lower and upper inequalities are the exact inf-sup and
  continuity bounds, respectively.
\end{assumption}

\begin{theorem}[Quasioptimality]
  \label{thm:quasiopt}
  Assumption~\ref{asm:wellposed} $\implies$ the ideal PG
  method~\eqref{eq:ipg} is uniquely solvable for $x_h$ and
  \[
    \| x- x_h \|_X \le \frac{C_2}{C_1} \inf_{z_h\in X_h } \| x - z_h \|_X
  \]
  where $x$ is the unique exact solution of~\eqref{eq:weakform}.
\end{theorem}
\begin{proof}
  We want to apply Theorem~\ref{thm:BNBapprox}. To this end, first
  observe that $T$ is injective: Indeed, if $T z=0$, then
  by~\eqref{eq:T}, we have $b(z,y)=0$ for all $y\in Y$, so
  Assumption~\ref{asm:wellposed} implies that $z=0$. Thus $\dim(X_h) =
  \dim(\Yhopt)$.

  Furthermore, if the inf-sup condition of
  Assumption~\ref{asm:wellposed} holds, then it is an exercise (see
  Exercise~\ref{x:C1} below) to show that the other inf-sup condition,
  \begin{equation}
    \label{eq:dinfsup}
  C_1 \| z \|_Y \le 
  \sup_{0\ne y \in Y} 
  \frac{| b(z,y)|}{ \| y \|_Y } \qquad \forall z\in X,
  \end{equation}
  holds with the {\em same} constant $C_1$. This, together with
  Proposition~\ref{prop:inf-sup-follows} shows that the discrete
  inf-sup condition~\eqref{eq:discrete-inf-sup} holds with the same
  constant. Hence Theorem~\ref{thm:BNBapprox} gives the result.
\end{proof}

\begin{exercise}
  \label{x:Aadjinv}
  Suppose $Z_1,Z_2$ are Banach spaces and $A: Z_1 \to Z_2$ is a linear 
  \cts\ bijection. Then prove that the inverse of its dual $(A')^{-1}$
  exists, is continuous, and $ (A')^{-1} = (A^{-1})'.$
\end{exercise}

\begin{exercise}
  \label{x:C1}
  Prove that, under the assumptions of Theorem~\ref{thm:BNB}, if
  Statement~(\ref{item:bnb3}) of Theorem~\ref{thm:BNB} holds for some
  $C_1$, then Statement~(\ref{item:bnb2}) also holds with the same
  constant $C_1$. ({\em Hint:} Use~Exercise~\ref{x:Aadjinv} but do not
  forget that our spaces are over $\CCC$.)
\end{exercise}

\begin{definition}
  \label{def:energy-norm}
  Let $R_Y : Y \to Y^*$ denote the {\em Riesz map} defined by $ (R_Y
  y)(v) = (y,v)_Y, $ for all $ y$ and $v$ in $Y$. It is well known to
  be invertible and isometric:
  \begin{equation}
    \label{eq:RY}
    \| R_Y    y \|_{Y^*} = \| y \|_Y. 
  \end{equation}
  Let $B: X \to Y^*$ be the {\em operator generated by the form}
  $b(\cdot,\cdot)$, i.e., $Bx(y) = b(x,y)$ for all $x\in X$ and $y\in
  Y$.  By the definition of $T$ in~\eqref{eq:T}, it is obvious that
  \begin{equation}
    \label{eq:TB}
    T = R_Y^{-1} \circ B.
  \end{equation}
  Finally, for any $z \in X$, we define the {\em energy norm} of $z$
  by $ \ntrip{z }_X \defn \| T z\|_Y$.  Clearly, by
  Proposition~\ref{prop:optimalopti},
  \[
  \ntrip{z }_X  = \| T z\|_Y
  = \sup_{0\ne y \in Y} 
  \frac{| b(z,y)|} { \|y \|_Y}.
  \]
\end{definition}

\begin{exercise}
  Prove that if Assumption~\ref{asm:wellposed} holds, then
  $\ntrip{\cdot}_X$ and $\| \cdot\|_X$ are equivalent norms:
  \[
  C_1 \| z \|_X \le \ntrip{z}_X \le C_2 \| z \|_X \qquad \forall z\in X.
  \]
\end{exercise}

\begin{theorem}[Residual minimization]
  \label{thm:dpgleastsq}
  Suppose Assumption~\ref{asm:wellposed} holds and $x$
  solves~\eqref{eq:weakform}.  Then, \tfae\ statements:
  \begin{enumerate}[\quad i)\;]
  \item \label{item:dpgls-1} $x_h \in X_h$ is the unique solution of
    the ideal PG method~\eqref{eq:ipg}.
  \item \label{item:dpgls-2} $x_h$ is the best approximation to $x$
    from $X_h$ in the following sense:
    \[
    \ntrip{x-x_h}_X = \inf_{z_h\in X_h}\ntrip{ x - z_h}_X
    \]
  \item \label{item:dpgls-3} $x_h$ minimizes residual in the
    following sense:
    \[
    x_h = \aarg \min_{z_h \in X_h} \| \ell - B z_h \|_{Y^*}.
    \]
  \end{enumerate}
\end{theorem}
\begin{proof}
  $(\ref{item:dpgls-1}) \iff~(\ref{item:dpgls-2}):$ 
  \begin{align*}
    x_h \text{ solves~\eqref{eq:ipg}} 
    & \iff b( x- x_h, y_h)=0  && \forall y_h \in \Yhopt
    \\
    & \iff b( x- x_h, Tz_h)=0 && \forall z_h \in X_h
    \\
    & \iff (T( x- x_h), Tz_h)_Y=0 && \forall z_h \in X_h,
  \end{align*}
  and the result follows since $(T\cdot,T\cdot)_Y$ is the inner
  product generating the $\ntrip{\cdot}_X$ norm.

  $(\ref{item:dpgls-2}) \iff~(\ref{item:dpgls-3}):$
  \begin{align*}
    \ntrip{x-x_h}_X = \inf_{z_h\in X_h}\ntrip{ x - z_h}_X
    & \iff \| T(x-x_h)\|_Y = \inf_{z_h\in X_h} \|T( x - z_h)\|_Y
    \\
    & \iff \| R_Y^{-1} B(x-x_h)\|_Y = \inf_{z_h\in X_h} \|R_Y^{-1} B( x - z_h)\|_Y,
    && \text{by~\eqref{eq:TB},}
    \\
    & \iff \| B(x-x_h)\|_{Y^*} = \inf_{z_h\in X_h} \| B( x - z_h)\|_{Y^*},
    && \text{by~\eqref{eq:RY},}
    \\
    & \iff \| \ell-Bx_h\|_{Y^*} = \inf_{z_h\in X_h} \| \ell - B z_h\|_{Y^*}
  \end{align*}
  since $\ell = Bx$. This proves the result.
\end{proof}

\begin{definition}
  \label{def:errorrep}
  Let $x$ solve~\eqref{eq:weakform} and $x_h$ solve~\eqref{eq:ipg}.
  We call $\veps = R_Y^{-1}(\ell - B x_h)$ the {\em error
    representation function}. Clearly, 
  \[
  \| \veps \|_Y = \| R_Y^{-1} B ( x- x_h) \|_Y = \| T(x-x_h)\|_Y = \ntrip{ x-x_h}_X,
  \]
  i.e., the $Y$-norm of $\veps$ measures the error in the energy
  norm. Note that $\veps$ is the unique element of $Y$ satisfying
  \[
  (\veps,y)_Y = \ell(y) - b(x_h,y), \qquad \forall y\in Y.
  \]
\end{definition}

\begin{theorem}[Mixed Galerkin formulation]
  \label{thm:mixed}
  \Tfae\ statements:
  \begin{enumerate}[\quad i)\;]
  \item \label{item:dpgmixed-1} $x_h \in X_h$ solves the ideal PG
    method~\eqref{eq:ipg}.
  \item \label{item:dpgmixed-2} $x_h$ and $\veps$ solves the mixed
    formulation
    \begin{subequations}
      \begin{align}
        \label{eq:6}
        (\veps,y)_Y + b(x_h, y) & = \ell( y) &&\forall y\in Y, \\
        \label{eq:7}
        b(z_h,\veps)            & = 0      &&\forall z_h \in X_h.
      \end{align}
    \end{subequations}
  \end{enumerate}
\end{theorem}
\begin{proof}
  $(\ref{item:dpgmixed-1}) \implies~(\ref{item:dpgmixed-2}):$
  Since~\eqref{eq:6} holds by the definition of the error
  representation function, we only need to prove~\eqref{eq:7}. To this
  end, $b(z_h,\veps) = (Tz_h, \veps)_Y= (Tz_h, R_Y^{-1} (\ell - Bx_h))_Y
  = (Tz_h, T(x-x_h))_Y$, which being the conjugate of $b(x-x_h,
  Tz_h)$, vanishes.

  $(\ref{item:dpgmixed-2}) \implies~(\ref{item:dpgmixed-1}):$
  Since~\eqref{eq:6} implies $ b(x_h, y_h) = \ell(y_h) -
  (\veps,y_h)_Y$ for all $y_h \in \Yhopt$, it suffices to prove that
  $(\veps, y_h)_Y=0$ for all $y_h\in \Yhopt$: But this is obvious from 
  \begin{align*}
    (T z_h, \veps)_Y & = b( z_h,\veps) =0 &&\forall z_h \in X_h,
  \end{align*}
  which holds by virtue of~\eqref{eq:7}.
\end{proof}

To summarize the theory so far, we have shown that there are test
spaces that can pair with any given trial space to generate an ideal
Petrov-Galerkin method that is guaranteed to be stable. Moreover, the
discrete inf-sup constant of the method is the same as the
exact inf-sup constant. We then showed, in
Theorem~\ref{thm:dpgleastsq}, that the resulting methods are least
square methods that minimize the residual in a dual norm. Finally, we
also showed that the method can be interpreted as a (standard
Galerkin rather than a Petrov-Galerkin) mixed formulation on 
$Y\times X_h$, after introducing an error representation function.

\begin{definition}
  \label{def:iDPG}
  Suppose an open $\om \subset \RRR^N$ is partitioned into disjoint
  open subsets $K$ (called {\em elements}), forming the collection
  $\oh$ (called {\em mesh}), such that the union of $\bar K$ for all
  $K\in\oh$ is $\bar \om.$ Let $Y(K)$ denote a Hilbert space of some
  functions on $K$ with inner product $(\cdot,\cdot)_{Y(K)}$. An
  {\bf{\em ideal DPG method}} is an ideal PG method with
  \begin{equation}
    \label{eq:1}
    Y = \prod_{K\in\oh} Y(K), 
  \end{equation}
  endowed with the inner product 
  \[
  (y,v)_Y = \sum_{K\in \oh} (y|_K,v|_K)_{Y(K)} \qquad \forall y,v\in K,
  \]
  where $y|_K$ denotes the $Y(K)$-component of any $y$ in the product
  space $Y$.  For example, the space $Y=H^1(\om)$ is not of the
  form~\eqref{eq:1}, while $Y = H^1(\oh) \defn \{ v \in L^2(\om): v|_K
  \in H^1(K)$ for all $K\in\oh\}$ is of the form~\eqref{eq:1}.
\end{definition}

Such DPG methods are interesting due to the resulting
locality of $T$. Note that to compute a basis for the optimal test
space, we must solve~\eqref{eq:T} to compute $Tz$ for each $z$ in a
basis of $X_h$.  That equation, namely $(Tz,y)_Y = b(z,y)$, decouples
into independent equations on each element, if $Y$ has the
form~\eqref{eq:1}. Indeed, the component of $Tz$ on an element $K$,
say $t_K = Tz|_K$ can be computed (independently of other elements) by
solving $ (t_K, y_K)_{Y(K)} = b(z,y_K) $ for all $y_K\in Y(K).$ The
adjective {\bf{\em{discontinuous}}} in the name ``DPG'' should no longer
be a surprise since test spaces $Y$ of the form~\eqref{eq:1} admit
(discontinuous) functions with no continuity constraints across
element interfaces.

\section{Examples and Connections}

Although we presented the theory over $\CCC$ for generality (e.g., to
cover harmonic wave propagation), it continues to apply for
real-valued bilinear and linear forms (in place of sesquilinear and
conjugate-linear forms). All the examples in this section are over
$\RRR$.

\begin{example}[Standard FEM]
  \label{eg:std}
  Set 
  \begin{align*}
    (v,w)_Y & = \int_\om \grad v\cdot\grad w \; , 
    \quad
    X  = Y = H_0^1(\om), \quad    \| u \|_X^2 = \| u\|_Y^2 = (u,u)_Y,
  \end{align*}
  and consider the standard weak formulation of the Dirichlet
  problem: For any given $F\in H^{-1}(\om)$, find $u \in H_0^1(\om)$ solving 
  \[
  b(u,v) = F(v), \qquad \forall v \in H_0^1(\om),
  \]
  where $b(u,v) = (u,v)_Y$. Clearly, in this case, the trial-to-test
  operator $T$ is the identity map $I$ on $ X$. Hence, if we set $X_h$ to
  the standard Lagrange finite element subspace of $H_0^1(\om)$ based
  on a finite element mesh, we get $\Yhopt = X_h$. Thus {\em the
    standard finite element method uses an optimal test space.} Since
  the form is coercive, it obviously satisfies
  Assumption~\ref{asm:wellposed}, so the previously discussed theorems
  apply for this method.
\end{example}

\begin{example}[$L^2$-based least-squares method]
  \label{eg:L2}
  Suppose $X$ is a Hilbert space and $A: X \to L^2(\om)$ is a
  continuous bijective linear operator. Then setting $Y=L^2(\om)$, the
  problem of finding a $u\in X$ such that $Au = f$, for any given
  $f\in Y$, can be put into a variational formulation by setting
  \[
  b(u,v) = (Au, v)_Y, \qquad \ell(v) = (f,v)_Y.
  \]
  Then it is obvious that 
  $
  T u = A u,
  $
  so $\Yhopt = A X_h$.  It is also easy to verify that
  Assumption~\ref{asm:wellposed} holds: By the bijectivity of $A$,
  \begin{align*}
    \text{uniqueness: }
    &
    z\in X, b(z,y)=0 \;\forall\; y\in L^2(\om) \implies A z=0 \implies z=0,
    \\
    \text{inf-sup: }
    & \sup_{0\ne z\in X} \frac{ | (Az,y)_Y |}{\| z\|_X} \ge 
    \frac{ | (y,y)_Y |}{\| A^{-1} y\|_X} \ge C_1 \| y\|_Y,
  \end{align*}
  with $C_1 = \| A^{-1} \|^{-1}$.  Hence, Theorems~\ref{thm:quasiopt}
  and~\ref{thm:dpgleastsq} hold for this method.

  Finally, since $B = R_{L^2(\om)} A$ and $\ell = R_{L^2(\om)} f$, by
  the isometry of the Riesz map, the residual minimization property of
  Theorem~\ref{thm:dpgleastsq}(\ref{item:dpgls-3}) implies that for
  any trial subspace $X_h \subset X$,  we have
  \[
    x_h = \aarg \min_{z_h \in X_h} \| f  - A z_h \|_{L^2(\om)},
  \]
  i.e., {\em in this example, the ideal PG method coincides with the
    standard $L^2(\om)$-based least-squares method.}
\end{example}

\begin{example}[1D o.d.e.\ without integration by parts]
  \label{eg:odeWOip}
  Let $\om = (0,1)$, $f \in L^2(\om)$.  Consider the boundary value
  problem (where primes denote differentiation) to find $u(x)$ solving
  \begin{subequations}
    \label{eq:1Dode}
    \begin{align}
      \label{eq:1Dode-a}
      u' & = f && \text{ on } (0,1), \\
      \label{eq:1Dode-b}
      u(0) & = 0 && \text{(boundary condition at $x=0$)}.
    \end{align}
  \end{subequations}
  A Petrov-Galerkin variational formulation is immediately obtained by
  multiplying~\eqref{eq:1Dode-a} with an $L^2$ test function $v$ and
  integrating. The resulting forms are 
  \[
  b(u,v) = \int_0^1 u' v, \quad \ell(v) = \int_0^1 f v,
  \]
  and the spaces are 
  \begin{align*}
    X = \{ u \in H^1(0,1): \; u(0)=0\},
    \quad 
    Y=L^2(0,1).
  \end{align*}
  This now fits into the framework of Example~\ref{eg:L2}. (Indeed, the operator 
  \begin{equation}
    \label{eq:4}
    A u = u', \quad A:    X \to Y \text{ is a bijection,}
  \end{equation}
  because for any $f\in Y$, the function $u (x)= \int_0^x f(s)\; ds $
  is in $X$ and satisfies $Au =f$.) Hence, as already discussed in
  Example~\ref{eg:L2}, this method reduces to a standard $L^2$-based
  least-squares method.
\end{example}
 
\begin{example}[1D o.d.e.\ with integration by parts]
\label{eg:odeWIP}
  We consider the same boundary value problem as above,
  namely~\eqref{eq:1Dode}, but now develop a different variational
  formulation for it. Multiply~\eqref{eq:1Dode-a} by a test function
  $v \in C^1(\bar\om)$ and integrate by parts to get 
  \[
  - \int_0^1 u v' + u(1) v(1) - u(0) v(0)  = \int_0^1 f v
  \]
  Using~\eqref{eq:1Dode-b} and letting the unknown value $u(1)$ to be
  a separate variable $\hat u_1$, to be determined, we have derived the
  variational equation
  \[
  - \int_0^1 u v' + \hat u_1 v(1)  = \int_0^1 f v.
  \]
  We let the pair $(u,\hat u_1)$ to be a group variable $z$, and fix
  an appropriate functional setting.  Set the forms by
  \[
  b(z, v) \equiv  b( \,(u,\hat u_1), v) = 
  \hat u_1 v(1) - \int_0^1 u v', 
  \qquad \ell(v) = \int_0^1 f v,
  \]
  the spaces by 
  \[
  X = L^2(\om) \times \RRR,\quad Y = H^1(\om), \quad \text{where }\om=(0,1),
  \]
  and the norms by 
  \begin{align}
    \nonumber
    \| z \|_X^2 
    & \equiv \| (u,\hat u_1) \|_X^2 = \| u \|_{L^2(\om)}^2 + | \hat u_1|^2
    \\ \label{eq:8}
    \| v \|_Y^2 
    & = \| v'\|_{L^2(\om)}^2 + |v(1)|^2.
  \end{align}
  By Sobolev inequality, $v(1)$ makes sense for $v\in Y$, so the above
  set $b(\cdot,\cdot)$ and $\|\cdot\|_Y$ are well-defined.  In fact
  (by Exercise~\ref{x:norm-eq-1d}), the above set norm $\| v \|_Y$ is
  equivalent to the standard $H^1(\om)$ norm.

  Next, let us verify Assumption~\ref{asm:wellposed}. First, suppose
  $(u,\hat u_1)$ satisfies 
  \begin{equation}
    \label{eq:9}
      b(\, (u,\hat u_1), v)=0\quad \forall  v\in Y.
  \end{equation}
  Then, choosing $v \in \DD(\om)$, the set of infinitely
  differentiable compactly supported functions in $\om$, we find
  that the distributional derivative $u'$ vanishes. Hence $u\in
  H^1(\om)$. Going back to a general $v\in Y$, we may now integrate
  equation~\eqref{eq:9} by parts to obtain
  \begin{equation}
    \label{eq:5}
    -u(1) v(1) + u(0) v(0) +\hat u_1 v(1)
    =0.
  \end{equation}
  Choosing $v(x)=1-x$, we obtain $u(0)=0$. Thus, $u$
  solves~\eqref{eq:1Dode} with zero data, so $u\equiv 0$
  by~\eqref{eq:4}. From~\eqref{eq:5} we also have $u(1) = \hat u_1$,
  which together with $u\equiv 0$ implies
  \begin{equation}
    \label{eq:10}
    \hat u_1 =0, \qquad u =0.
  \end{equation}
  Thus the uniqueness part of Assumption~\ref{asm:wellposed}, $\{ z
  \in X: \;b(z, y)=0,\; \forall y\in Y\} = \{0\}$ holds.

  For the remaining parts, we begin by noting that by Cauchy-Schwarz
  inequality,
  \begin{align*}
    \sup_{z\in X} \frac{ |b( z, v) |^2 }{ \| z\|_X^2 } 
     & =
     \sup_{(u,\hat u_1)\in X}
     \frac{ \displaystyle{\left|
           \hat u_1 v(1)
           - \int_0^1 u v' \right|^2 }}
     { \| u\|_{L^2(\om)}^2 + | \hat u_1 |^2}
     \le 
     \sup_{(u,\hat u_1)\in X}
     \frac{ \displaystyle{
           \left( |\hat u_1|^2 + \| u \|_{L^2(\om)}^2\right) \| v \|_Y^2 }}
     { \| u\|_{L^2(\om)}^2 + | \hat u_1 |^2}
     =
     \| v \|_Y^2
  \end{align*}
  while on the other hand,  given any $v\in Y$, 
  choosing $u = -v'$ and $\hat u_1 = v(1)$, we get 
  \begin{align*}
    \sup_{z\in X} \frac{ |b( z, v) |^2 }{ \| z\|_X^2 } 
     & =
     \sup_{(u,\hat u_1)\in X}
     \frac{ \displaystyle{\left|
           \hat u_1 v(1) - \int_0^1 u v' \right|^2 }}
     { \| u\|_{L^2(\om)}^2 + | \hat u_1 |^2}
     \ge
     \frac{ \displaystyle{\left|
            | v(1)|^2  +  \int_0^1 | v'|^2 \right|^2 }}
     { \| -v'\|_{L^2(\om)}^2 + | v(1) |^2}
     =\| v\|_Y^2
  \end{align*}
  Thus Assumption~\ref{asm:wellposed} holds with $C_1=C_2=1$.

  We can now calculate the optimal test space (see
  Exercise~\ref{x:ode} below) for any given trial space. Let
  $P_p(\om)$ denote the space of polynomials of degree at most $p$ on
  $\om$. We experiment with
  \[
  X_h = P_p(\om) \times \RRR,
  \]
  i.e., the discrete solution $x_h = (u_h,\hat u_{1,h})$ has $u_h$ in
  $P_p(\om) \subset L^2(\om)$ and the point flux value approximation
  $\hat u_{1,h}$ in $\RRR$. The resulting method was implemented in
  \fenics\ (code can be downloaded from
  \href{http://web.pdx.edu/~gjay/pub/1Dpgode.py}{\underline{here}}).
  Collecting the results obtained with an $f$ corresponding to an
  exact solution with a sharp layer,
  \[
  u = \frac{ e^{M(x-1)} - e^{-M}} { 1 - e^{-M}},
  \]
  we obtain Figure~\ref{fig:spectralPG}.

  The first graph in Figure~\ref{fig:spectralPG} plots the exact
  solution $u$ and the computed $u_h$ for three values of $p$. We also
  implemented the method of Example~\ref{eg:odeWOip} and plotted the
  corresponding solutions in the next graph in
  Figure~\ref{fig:spectralPG}. Comparing, we find that the ideal PG
  method of the current example performs better than that of
  Example~\ref{eg:odeWOip}.  Finally, we also plotted the
  $L^2(\om)$-projections of the exact solution on $P_p(\om)$ in the
  last graph in Figure~\ref{fig:spectralPG}. Comparing the plots, the
  first and the third figures appear
  identical. Exercise~\ref{ex:proj1D} asks you to show that this is
  indeed the case.
  \begin{figure}
    \centering
    \includegraphics[width=\textwidth]{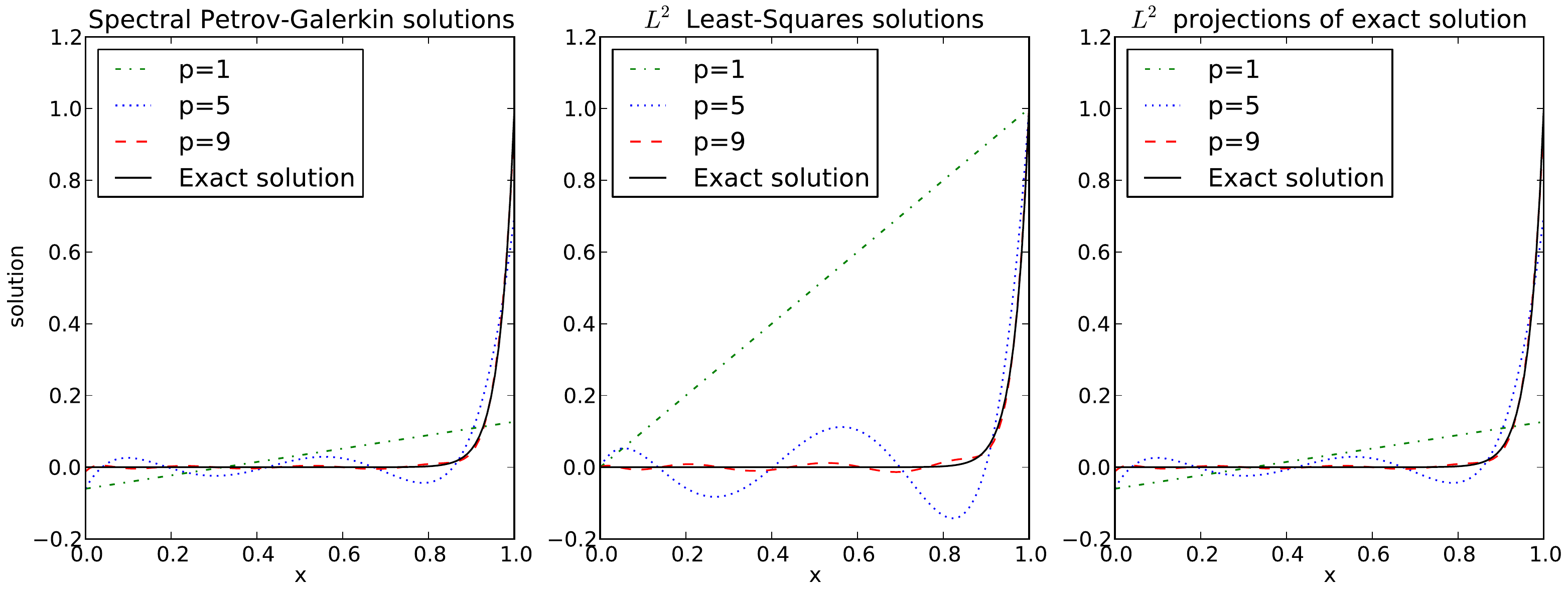}
    \caption{Solutions from one-element one-dimensional computations}
    \label{fig:spectralPG}
  \end{figure}
\end{example}

\begin{exercise}
  \label{x:norm-eq-1d}
  Prove that the norm defined in~\eqref{eq:8} is equivalent to the
  standard Sobolev norm defined by $\| v \|^2_{H^1(\om)} = \|
  v\|_{L^2(\om)}^2 + \| v'\|_{L^2(\om)}^2$. {\em Hint:} Use a Sobolev
  inequality and a \Poincare-type inequality.
\end{exercise}
\begin{exercise}
  \label{x:ode}
  Prove that in the setting of Example~\ref{eg:odeWIP}, an explicit
  formula for $T(u,\hat u_1)$ can be given for any $(u,\hat u_1) \in
  X$:
  \begin{equation}
    \label{eq:3}
    T (u,\hat u_1) = \hat u_1 + \int_x^1 u(s) \, ds.
  \end{equation}  
  Next, use~\eqref{eq:3} to prove that if
  $X_h = P_p(\om) \times \RRR$, then $\Yhopt = P_{p+1}(\om)$.
\end{exercise}

\begin{exercise}
  \label{ex:proj1D}
  Prove that the $u_h$ resulting from the ideal Petrov-Galerkin method
  of Example~\ref{eg:odeWIP} equals the $L^2(0,1)$ projection of $u$
  and that $\hat u_{1,h}=\hat u_1$. {\em Hint:} Apply
  Theorem~\ref{thm:quasiopt}.
\end{exercise}

\begin{exercise}
  Suppose $\om=(0,1)$ is partitioned by the mesh $0=x_0<x_1<\cdots <
  x_m=1$.  Consider the method of Example~\ref{eg:odeWIP}, modified to
  use the different trial subspace $X_h=\{ u: u|_{(x_{i-1},x_i)}\in
  P_p(x_{i-1},x_i),$ for $i=1,\ldots, m\} \times \RRR$.  Show that $T$
  does not, in general, map locally supported trial functions to
  locally supported test functions, by exhibiting a $(u,\hu_1) \in X_h$
  such that $\supp(u) \subseteq [x_{i-1},x_i]$ for some $i$ but
  $\supp( T(u,\hu_1)) \not\subseteq [x_{i-1},x_i]$.
\end{exercise}

\begin{example}[An ideal DPG method]
  \label{eg:1d-IDPG}
  Continuing to consider~\eqref{eq:1Dode}, we now sketch how to extend
  the ideal PG scheme of Example~\ref{eg:odeWIP} to an ideal DPG
  scheme. Following the setting of Definition~\ref{def:iDPG}, we
  assume $\om=(0,1)$ is partitioned into $\oh$ consisting of $m$ intervals
  $(x_{i-1},x_i)$ for all $i=1,\ldots, m$, with $x_0=0$ and
  $x_m=1$. Let $v^\pm(x)$ denote the limiting value of $v$ at $x$ from
  the right and left, respectively. Set
  \begin{align*}
    Y & = H^1(\oh)= \{z\in L^2(\om) : z|_K\in H^1(K),\;\forall K\in\oh\}
    \\
    \| y\|_Y^2 & = \sum_{i=1}^m 
    \left(|y^-(x_i)|^2+\int_{x_{i-1}}^{x_i} |y'|^2 \right)
    \\
    X & = L^2(\om)\times \RRR^m
    \\
    \| (u,{\hat u}_1,{\hat u}_2,\ldots,{\hat u}_m)\|_X^2
    & = \| u\|_{L^2(\om)}^2 + |\hat u_1|^2 + |\hat u_2|^2 +\cdots+ |\hat u_m|^2
    \\
    \ell(y)
    & = (f,y)_{L^2(\om)}
    \\
    b(\,(u,{\hat u}_1,{\hat u}_2,\ldots,{\hat u}_m),\,y) 
    & =     \sum_{i=1}^m\left(
    {\hat u}_i y^-(x_i) - {\hat u}_{i-1} y^+(x_{i-1})-
    \int_{x_{i-1}}^{x_i} u y'\right),
  \end{align*}
  with the understanding that ${\hat u}_0=0$. Note that if $m=1$, then
  this reduces to the method of Example~\ref{eg:odeWIP}.  For general
  $m$, the action of the trial-to-test operator $T$ is local and can
  be computed element by element (see Exercise~\ref{x:T-idode}). The
  method for general $m$ can be analyzed as in Example~\ref{eg:odeWIP}
  (see Exercise~\ref{x:analyze1DIPG}).
\end{example}

\begin{exercise}
  \label{x:T-idode}
  Prove that, in the setting of Example~\ref{eg:1d-IDPG},
  \[
  T(0,\cdots,0,\hat u_i,0,\cdots,0)=\hat u_i \times 
  \left\{
  \begin{aligned}
    & 1
    &&\text{ if } x \in (x_{i-1},x_i) \\
    & x - x_{i+1} -1
    &&\text{ if } x \in (x_i,x_{i+1}) \\
    & 0            
    &&\text{ elsewhere.}
  \end{aligned}
\right.
  \]
\end{exercise}

\begin{exercise}
\label{x:analyze1DIPG}
Verify Assumption~\ref{asm:wellposed} for the formulation of
Example~\ref{eg:1d-IDPG}.
\end{exercise}

The process by which we extended the formulation of
Example~\ref{eg:odeWIP} to that in Example~\ref{eg:1d-IDPG} is an
instance of ``hybridization''. Variables like $\hat u_i$ in
Example~\ref{eg:1d-IDPG} are referred to by various names such as
facet, or inter-element, or {\bf{\em interface unknowns}}, and in
the DG community, by names like {\em numerical fluxes} or numerical traces.
To put the hybrid method in a more general PG context, we use the
abstract setting stated next.

\begin{assumption}  \label{asm:hybrid}
  Suppose $X$ takes the form $X_0 \times \hat X$ where $X_0$ and $\hX$
  are two Hilbert spaces and let the finite-dimensional subspace $X_h$
  have the form $\Xoh \times \hX_h$ with subspaces $\Xoh \subset X_0$
  and $\hX_h\subset \hX$.  Suppose there are continuous sesquilinear
  forms $\hat b(\cdot,\cdot): \hX \times Y \to \CCC$ and
  $b_0(\cdot,\cdot): X_0 \times Y \to \CCC$, in terms of which
  $b(\cdot,\cdot)$ is set by
  \[
  b( \,(u,\hu), y\,) = b_0( u, y) + \hb(\hu, y), 
  \]
for all $(u,\hu)\in X$ and $y\in Y$, and suppose 
\begin{equation}  \label{eq:21}
\Yo = \{ y\in Y : \hb(\hat u_h, y)=0,\; \forall \hat u_h
\in \hat X_h\}  
\end{equation}
is a closed subspace of $Y$. In addition to the already defined $T:
X\to Y$, define $T_0: X_0\to \Yo$ by $ (T_0 u, y)_Y = b_0(u,y)$, for
all $ y\in \Yo.$
\end{assumption}

\begin{subequations}
  \label{eq:2methods}
Under this setting, we consider two ideal PG methods:
\begin{align}
  \label{eq:15}
  \text{Find }(x_h,\hx_h)\in X_h:
  &\quad
  b(\,(x_h,\hx_h), y\,) = \ell(y) &&  \forall y\in \Yhopt \equiv T(X_h).
  \\
  \label{eq:16}
  \text{Find } x_h\in \Xoh:
  &\quad
  b_0(x_h, y) = \ell(y) &&  \forall y\in \Yohopt \equiv T_0(\Xoh).
\end{align}
\end{subequations}
The interest in the ``hybridized'' form~\eqref{eq:15} arises because,
when moving from $\Yo$ to $Y$, one can often obtain test spaces of the
form in Definition~\ref{def:iDPG}, which make $T$ local.  This will
become clearer in Example~\ref{eg:comp}, discussed after the next
theorem, and later in Example~\ref{eg:laplacehybrid}.

\begin{theorem}[Hybrid method]
  \label{thm:hybrid}
  Suppose Assumption~\ref{asm:hybrid} holds. Then, the test spaces
  in~\eqref{eq:2methods} satisfy $\Yohopt \subset\Yhopt$. Hence,
  \[
  (x_h,\hx_h) \in X_h\text{ solves~\eqref{eq:15}}
  \implies  x_h \text{ solves~\eqref{eq:16}.}   
  \]
\end{theorem}
\begin{proof}
  Since $\Yhopt$ is a closed subspace of $Y$, we have the orthogonal
  decomposition
  \begin{equation}
    \label{eq:17}
    Y = \Yhopt + Y_\perp
  \end{equation}
  where $Y_\perp$ is the $Y$-orthogonal complement of $\Yhopt$.  Let
  $y_0\in \Yohopt$. Apply~\eqref{eq:17} to decompose $y_0 = y_h +
  y_\perp$, with $y_h\in \Yhopt$ and $y_\perp \in Y_\perp$.

  First, we claim that $y_\perp\in \Yo$. This is because
  \begin{align*}
    \hb(\hu_h,y_\perp)=(T(0,\hu_h),y_\perp)_Y =0\qquad\forall \hu_h\in \hX_h.
  \end{align*}
  The last identity followed from the orthogonality of $y_\perp$ to
  $T(X_h)$.

  Next, we claim that $y_\perp=0$. It suffices to prove that
  $(y_0,y_\perp)_Y=0$ since $\| y_\perp\|_Y^2 =
  (y_0,y_\perp)_Y$. Since $y_0\in \Yohopt$, there is a $u_h\in X_h$
  such that $y_0=T_0u_h$. Then, 
  \begin{align*}
    (y_0,y_\perp)_Y
    & = (T_0 u_h, y_\perp)_Y  = b_0(u_h,y_\perp)
    && \text{as }y_\perp \in \Yo
    \\
    &=(T(u_h,0),y_\perp)_Y =0 
    && \text{as }T(X_h)  \perp y_\perp.
  \end{align*}
  Finally, since $y_\perp=0$, we have $y_0 = y_h+0\in \Yhopt.$ Thus
  $\Yohopt \subset \Yhopt$. The second statement of the theorem is now
  obvious by choosing $y\in \Yohopt$ in~\eqref{eq:15}.
\end{proof}

\begin{example}  \label{eg:comp}
  Set $\om=(0,1)$, $X_0=L^2(\om)\times \RRR, \hX=\hX_h=\RRR^{m-1},
  Y=H^1(\oh),$
  \begin{align*}
    b_0( \,(u,\hu_m), y)
    & =    
    {\hat u}_m y^-(1)  
    -\sum_{i=1}^m
      \int_{x_{i-1}}^{x_i} u y',
    \\
    \hb( \,(\hu_1,\ldots, \hu_{m-1}), y)
    & 
    =   
    \hu_1 y^-(x_1) -\hu_{m-1}y^+(x_{m-1}) + \sum_{i=2}^{m-1}\left(
    {\hat u}_i y^-(x_i) - {\hat u}_{i-1} y^+(x_{i-1})
  \right),
  \end{align*}
  Then, the method~\eqref{eq:15} yields the method of
  Example~\ref{eg:1d-IDPG}. It is easy to see that $ \Yo=H^1(\om)$.
  Hence the method~\eqref{eq:16} yields the method of
  Example~\ref{eg:odeWIP}. By Theorem~\ref{thm:hybrid}, the (global
  basis of) optimal test functions of Example~\ref{eg:odeWIP} can be
  expressed as a linear combination of the (local basis of) optimal
  test functions of Example~\ref{eg:1d-IDPG}.
\end{example}


\section{Inexact test spaces}

To compute the optimal test spaces, we need to apply $T$, which
requires solving~\eqref{eq:T}, typically an infinite-dimensional
problem. Although we have seen some examples where the action of $T$
can be computed in closed form, for the vast majority of interesting
boundary value problems, this is not feasible. Hence we are motivated
to substitute the optimal test functions by inexact (approximations
of) optimal test functions.

Let $Y^r$ denote a finite-dimensional subspace of $Y$ (with the index
$r$ related to its dimension.) Let $T^r: X \to Y^r$ be defined by $
(T^r w, y)_Y = b(w,y)$ for all $y \in Y^r$. In general, $T^r \ne T$.

\begin{definition}
  \label{def:DPG}
  A {\bf{\em DPG method}} for~\eqref{eq:weakform} uses a space $Y$ as
  in the ideal DPG method of Definition~\ref{def:iDPG},
  finite-dimensional subspaces $X_h \subset X$ and $Y^r \subset Y$,
  and computes $x_h$ in $X_h$ satisfying
    \begin{equation}
    \label{eq:pdpg}
      b({\xh},{y}) = \ell({y}),
      \qquad\forall {y} 
      \in Y_h^r \defn T^r(X_h).
  \end{equation}
  The DPG method is sometimes also called the ``practical'' DPG method,
  because it uses the inexact, but practically computable, test space
  $Y_h^r$ (in contrast to the ideal DPG method, which uses the exact
  optimal test space $\Yhopt$).
\end{definition}

\begin{assumption}
  \label{asm:Pi}
   There is a linear operator $\varPi: Y \to Y^r$ and a
  $C_\varPi>0$ such that for all $ {w_h}\in {X_h}$ and all $ v \in Y$,
  \[
    b(  w_h , v - {\varPi}  v) =0, \qquad \text{and} \qquad
    \| {\varPi}  v \|_Y \le C_{{\varPi}} \|  v \|_Y.
  \]
\end{assumption}

\begin{theorem}
  \label{thm:inex}
  Suppose Assumptions~\ref{asm:wellposed} and~\ref{asm:Pi} hold.  Then
  the DPG method~\eqref{eq:pdpg} is uniquely solvable for $x_h$ and
  \[
  \| x- x_h \|_X \le \frac{C_2C_\varPi}{C_1}
  \inf_{z_h\in X_h } \| x - z_h \|_X
  \]
  where $x$ is the unique exact solution of~\eqref{eq:weakform}.
\end{theorem}
\begin{proof}
  First, note that by Assumption~\ref{asm:Pi}, $T^r : X_h \to Y^r$ is
  injective: $T^r w_h =0$ for some $w_h\in X_h \implies b(w_h,y^r)=0$
  for all $y^r\in Y^r\implies b(w_h,\varPi y)=0$ for all $y\in Y
  \implies b(w_h,y) =0$ for all $y\in Y$, which by
  Assumption~\ref{asm:wellposed} implies that $w_h=0$. Thus,
  \[
    \dim( Y_h^r) = \dim(X_h).
  \]
  Next, for any $z_h \in X_h$, let
  \[
  s_0 = \sup_{0\ne y\in Y} \frac{ | b(z_h,y) |} {
    \| y \|_Y},
  \qquad 
  s_1 =  \sup_{0\ne y\in Y^r} \frac{ | b(z_h, y^r) |} {
    \|y^r \|_Y}, 
  \qquad 
  s_2 =
  \sup_{0\ne y\in Y_h^r} \frac{ | b(z_h, y_h^r) |} {
    \|y_h^r \|_Y}.
  \]
  The result will follow from Theorem~\ref{thm:BNBapprox} once we
  prove the discrete inf-sup condition 
  \begin{equation}
    \label{eq:24}
    C_1 C_\varPi^{-1} \| z_h \|_X
    \le s_2.
  \end{equation}

  We proceed to bound $\| z_h\|_X$ using $s_0$, then $s_1$, and
  finally $s_2$.  Assumption~\ref{asm:wellposed} implies (by
  Exercise~\ref{x:C1}) that the inf-sup condition
  \[
  C_1 \| z_h \|_X \le s_0
  \]
  holds.  Hence Assumption~\ref{asm:Pi} implies
  \begin{align*}
    C_1 \| z_h \|_X 
    & \le \sup_{0\ne y\in Y} \frac{ | b(z_h,y) |} {
            \| y \|_Y} 
     = \sup_{0\ne y\in Y} \frac{ | b(z_h,\varPi y) |} {
      \| y \|_Y}  &&\forall z_h\in X_h
    \\
    &\le  \sup_{0\ne y\in Y} \frac{ | b(z_h,\varPi y) |} {
      C_\varPi^{-1} \|\varPi y \|_Y}  
    \le  \sup_{0\ne y\in Y^r} \frac{ | b(z_h, y^r) |} {
      C_\varPi^{-1} \|y^r \|_Y}  
    &&\forall z_h\in X_h,
  \end{align*}
  i.e., we have proven the tighter inf-sup condition $C_1
  C_\varPi^{-1} \| z_h\|_X \le s_1$. To finish the proof
  of~\eqref{eq:24}, it only remains to tighten it further by proving
  that $s_1 \le s_2$. Analogous to Proposition~\ref{prop:optimalopti},
  $s_1$ is attained at $T^r z_h$, so
  \begin{align*}
    s_1 = 
    \frac{ (T^r z_h,T^rz_h )_Y}{ \| T^r z_h\|_Y}
    \le
    \sup_{0\ne y_h^r\in Y_h^r} \frac{ (T^r z_h, y_h^r)_Y}
    {\|y_h^r \|_Y}=s_2.
  \end{align*}
  This shows~\eqref{eq:24} and finishes the proof.
\end{proof}

\begin{remark}
  Although Theorem~\ref{thm:inex} has more hypotheses than 
  Theorem~\ref{thm:quasiopt},
  \[
  \text{Theorem~\ref{thm:inex}}\implies
  \text{Theorem~\ref{thm:quasiopt}.}
  \]
  Indeed, the ideal PG method is obtained by simply setting $Y^r =Y$,
  and in that case, the trivial operator $\varPi=I$ satisfies
  Assumption~\ref{asm:Pi} with $C_\varPi=1$. (Note that
  Theorem~\ref{thm:inex} holds if we use any closed subspace
  $Y^r\subset Y$ in~\eqref{eq:pdpg}, not only finite-dimensional
  $Y^r$.)
\end{remark}

\begin{exercise}[Necessity \& Sufficiency of Assumption~\ref{asm:Pi}]
  Suppose Assumption~\ref{asm:wellposed} holds.  If there is a
  $C_0>0$ such that for all $\ell \in (Y_h^r)^*$ there exists a
  unique $x_h \in X_h$ satisfying~\eqref{eq:pdpg} and moreover 
  \[
  \| x_h \|_X \le C_0  \| \ell \|_{(Y_h^r)^*},
  \]
  then the method~\eqref{eq:pdpg} is called
  {\em stable} and $C_0$ is the {\bf{\em{stability}}} constant of the
  method. Show that 
  \[
  \text{the method~\eqref{eq:pdpg} is stable}
  \iff \text{Assumption~\ref{asm:Pi} holds},
  \]
  and relate the stability constant to the other constants.
\end{exercise}


\begin{definition}[cf.~Definition~\ref{def:energy-norm}]
  Let $\ntrip{x}_r \defn \| T^r x\|_Y$.  Let $R_{Y^r} : {Y^r} \to
  {(Y^r)}^*$ be the Riesz map defined by $ (R_{Y^r} y)(v) = (y,v)_Y, $
  for all $ y$ and $v$ in ${Y^r}$.  By the definition of $T^r$, it is
  easy to see that
  \begin{equation}
    \label{eq:11}
    T^rw= R_{Y^r}^{-1} B w.
  \end{equation}
\end{definition}

\begin{theorem}[cf.~Theorem~\ref{thm:dpgleastsq}]
  \label{thm:pdpgleastsq}
  Suppose Assumptions~\ref{asm:wellposed} and~\ref{asm:Pi} hold and
  let $x$ solve~\eqref{eq:weakform}.  Then, \tfae\ statements:
  \begin{enumerate}[\quad i)\;]
  \item \label{item:pdpgls-1} $x_h \in X_h$ is the unique solution of
    the DPG method~\eqref{eq:pdpg}.
  \item \label{item:pdpgls-2} $x_h$ is the best approximation to $x$
    from $X_h$ in the following sense:
    \[
    \ntrip{x-x_h}_r = \inf_{z_h\in X_h}\ntrip{ x - z_h}_r
    \]
  \item \label{item:pdpgls-3} $x_h$ minimizes residual in the
    following sense:
    \[
    x_h = \aarg \min_{z_h \in X_h} \| \ell - B z_h \|_{{(Y^r)^*}}.
    \]
  \end{enumerate}
\end{theorem}
\begin{proof}
  Follow along the lines of proof of Theorem~\ref{thm:dpgleastsq} but
  use~\eqref{eq:11} instead of~\eqref{eq:TB}.
\end{proof}


\begin{definition}[cf.~Definition~\ref{def:errorrep}]
  \label{def:estimator}
  Let $x$ solve~\eqref{eq:weakform}.  We call $\veps^r =
  R_{Y^r}^{-1}(\ell - B x_h)$ the {\em error estimator} of an $x_h$ in
  $X_h$. It is easy to see that it is the unique element of $Y^r$
  satisfying $ (\veps^r,y)_Y = \ell(y) - b(x_h,y), $ for all $y \in
  Y^r$.
\end{definition}

\begin{theorem}[cf.~Theorem~\ref{thm:mixed}]
  \label{thm:pdggmixed}
  \Tfae\ statements:
  \begin{enumerate}[\quad i)\;]
  \item \label{item:dpgmixed-1} $x_h \in X_h$ solves the DPG
    method~\eqref{eq:pdpg}.
  \item \label{item:dpgmixed-2} $x_h\in X_h$ and $\veps^r\in Y^r$
    solve the mixed formulation
    \begin{subequations}
      \label{eq:mixeddpg}
      \begin{align}
        \label{eq:22}
        (\veps^r,y)_Y + b(x_h, y) & = \ell( y) &&\forall y\in Y^r, \\
        \label{eq:23}
        b(z_h,\veps^r)            & = 0      &&\forall z_h \in X_h.
      \end{align}
    \end{subequations}
  \end{enumerate}
\end{theorem}
\begin{proof}
  Follow along the lines of the proof of Theorem~\ref{thm:mixed}.
\end{proof}
  %


\begin{exercise}
  \label{x:orth}
  Prove that $\veps^r$ is $Y$-orthogonal to $Y_h^r$.
\end{exercise}

Next, recall the setting of Assumption~\eqref{asm:hybrid} with $ b(
\,(u,\hu), y\,) = b_0( u, y) + \hb(\hu, y).$ Analogous
to~\eqref{eq:21}, define
\begin{equation}
  \label{eq:30}
  Y_0^r = \{ y\in Y^r : \; \hb(\hat u_h, y)=0,\; \forall \hat u_h
  \in \hat X_h\}
\end{equation}
and let $T_0^r: X_0\to Y_0^r$ be defined by $(T_0^r u, y)_Y = b_0(u,y)$
for all $y\in Y_0^r.$ Then consider the corresponding DPG methods:
\begin{subequations}
\begin{align}
  \label{eq:26}
  \text{Find }(x_h,\hx_h)\in X_h:
  &\quad
  b(\,(x_h,\hx_h), y\,) = \ell(y) &&  \forall y\in Y_h^r \equiv T^r(X_h).
  \\
  \label{eq:27}
  \text{Find } x_h\in \Xoh:
  &\quad
  b_0(x_h, y) = \ell(y) &&  \forall y\in \Yohr \equiv T_0^r(\Xoh).
\end{align}
\end{subequations}

\begin{theorem}[cf.~Theorem~\ref{thm:hybrid}]
  \label{thm:pdpghybrid}
  Suppose Assumption~\ref{asm:hybrid} holds.
Then, the test spaces satisfy
  $\Yohr \subset Y_h^r$. Hence,
  \[
  (x_h,\hx_h) \in X_h\text{ solves~\eqref{eq:26}}
  \implies  x_h \text{ solves~\eqref{eq:27}.}   
  \]
\end{theorem}
\begin{proof}
  Proceed as in the proof of Theorem~\ref{thm:hybrid}, after replacing
  $Y$ by $Y^r$, and $Y_\perp$ by the orthogonal complement of
  $Y_h^r$ in $Y^r$.
\end{proof}



\begin{remark}[Some ways to implement DPG methods]\hfill
  \label{rem:impl}
  \begin{enumerate}
  \item \label{item:Amat} Choose a local basis for $X_h$, say
    $e_j$. Compute $v_i = T^re_i$ (usually precomputed on a fixed
    reference element and mapped to physical elements).  Then assemble
    the {\em square} matrix 
    \begin{equation}
      \label{eq:29}
      A_{ij} = b(e_j,v_i) 
    \end{equation}
    by usual finite element techniques and solve.
  \item \label{item:Bmat} Let $e_j$ be as in~item~(\ref{item:Amat})
    and additionally select a local basis for $Y^r$, say $y_i$.
    Assemble the {\em rectangular} (since $\dim X_h \le \dim Y^r$
    typically) matrix
    \[
    B_{ij} = b(e_j,y_i)
    \]
    and the (block-diagonal) Gram matrix $M_{lm} =
    (y_l,y_m)_Y$. (Again, their assembly can be done by precomputing
    element matrices on a fixed reference element and mapping to
    physical elements.) Then form the square matrix
    $
    A = B^t M^{-1} B.
    $
    It is easy to see that this matrix equals~\eqref{eq:29}, so we proceed 
    as in item~(\ref{item:Amat}).
  
  \item Let $e_j$ and $y_i$ be as in item~(\ref{item:Bmat}). Assemble
    the matrices of~\eqref{eq:mixeddpg} and
    solve. Since~\eqref{eq:mixeddpg} is a standard Galerkin
    formulation, not a Petrov-Galerkin formulation, this technique
    requires no further explanation.  We will opt for this method in
    the code in the next section.
    
  \end{enumerate}
\end{remark}

\begin{exercise}
  Suppose the basis $e_j$ and the matrix $A$ are as in
  Remark~\ref{rem:impl}(\ref{item:Amat}). Prove that
  Assumptions~\ref{asm:wellposed} and~\ref{asm:Pi} imply the spectral
  {\bf{\em{condition number}}} of $A$ satisfies
  \[
  \kappa(A) \le \frac{\lambda_1}{\lambda_0} \frac{C_2^2 C_\varPi^2}{C_1^2},
  \]
  where $\lambda_0,\lambda_1$ are positive numbers such that $
  \lambda_0 \| \vec\chi \|_{2}^2 \le \| x \|_X^2 \le \lambda_1 \|
  \vec \chi \|_{2}^2$ holds for all $x =\sum_j \chi_j e_j$ in $
  X_h$.
\end{exercise}

\section{The Laplacian}  \label{sec:laplacian}

Let $\om$ be a bounded connected open subset of $\RRR^N$ for any $N\ge
2$ with Lipschitz boundary $\d\om$. We focus on the simple boundary
value problem
\begin{subequations}
  \label{eq:bvp}
  \begin{align}
  -\Delta u   & = f && \text{ on } \om, \\
  u & =0 && \text{ on } \d \om.
  \end{align}
\end{subequations}
All functions are real-valued in this section. We assume we have a
mesh $\oh$ as in Definition~\ref{def:iDPG} and additionally assume
that $\d K$ is Lipschitz for all $K\in\oh$ (so that we may use trace
theorems on each element), but the shape of the elements is unimportant
for now.

To develop our PG formulation for~\eqref{eq:bvp}, we set the test
space by
\begin{gather*}
  Y
   =
  H^1(\oh) = \{ v: \; v|_K \in H^1(K), \;\forall K
  \in \oh\} \equiv \prod_{K\in \oh} H^1(K),
  \\
  (v,y)_Y
  = (v,y)_\oh + (\grad v,\grad y)_\oh.
\end{gather*}
Multiplying~\eqref{eq:bvp} by a $v\in Y$ and integrating by parts on
any element $K\in \oh$, we obtain
\begin{equation}
  \label{eq:18}
\int_K \grad u \cdot \grad v -\int_{\d K} (n\cdot\grad u)  v
=\int_K f v.  
\end{equation}
As usual, the integral over $\d K$ must be interpreted as a duality
pairing in $H^{1/2}(\d K)$ if~$u$ is not sufficiently regular.
Summing up~\eqref{eq:18} over all $K\in \oh$ and letting $n\cdot \grad
u$ be an independent unknown, denoted by $\hq_n$, we derive the PG
formulation.  To state it precisely, we use these notations: Let $
(r,s)_\oh = \sum_{K\in\oh} (r,s)_K$ where $(\cdot,\cdot)_D$, for any
domain $D$, denotes the $L^2(D)$-inner product and $
\ip{\ell,w}_{\d\oh} = \sum_{K\in\oh} \ip{\ell,w}_{1/2,\d K}$ where
$\ip{ \ell,\cdot}_{1/2,\d K}$ denotes the action of a functional
$\ell$ in $H^{-1/2} (\d K)$. The PG weak formulation finds $(u, \hq_n)
\in X$ satisfying
\begin{equation}
  \label{eq:Laplacian}
  (\grad u, \grad v)_\oh - \ip{ \hq_n, v}_{\d\oh} = (f,v)_\om,
  \qquad \forall v \in Y,
\end{equation}
where the trial space $X$ is defined as follows: First, 
define the element-by-element trace  operator $\trn$ by
\[
\trn  : \Hdiv\om \to  \prod_{K\in\oh}  H^{-1/2}(\d K), 
\qquad \trn r|_{\d K} =
r\cdot n |_{\d K}.
\]
Here and throughout, $n$ generically denotes the unit outward normal
of any domain under consideration. Now set 
\begin{equation}
  \label{eq:13} 
\begin{aligned}
  H^{-1/2}(\d \om_h)
  & =  \ran(\trn),
  \\
  \|  \hr_n \|_{H^{-1/2}(\d\oh)} 
  & = \inf\bigg\{ \| q \|_{\Hdiv\om}: \;\forall q \in \Hdiv\om
  \text{ such that } \trn(q) = \hr_n\bigg\}.
\end{aligned}
\end{equation}
The trial space is then given by
\begin{align*}
  X & = H_0^1(\om) \times H^{-1/2} (\d\oh),
  \\
  \| ( w,\hr_n) \|_X^2
  & = \| \grad w \|_{L^2(\om)}^2 + \| \hr_n \|_{H^{-1/2}(\d\oh)}^2.
\end{align*}
In~\eqref{eq:13}, the norm is a quotient norm (see
Exercises~\ref{x:quotient}--\ref{x:Ext}). With this quotient norm,
we will not need to explicitly use the subspace topology inherited from
$\prod_K H^{-1/2}(\d K)$.

\begin{exercise}
  \label{x:quotient}
  Suppose $X_1$ and $X_2$ are linear spaces, $A: X_1 \to X_2$ is a
  linear onto map, and let $\pi: X_1 \to X_1 /\ker A$ be the quotient
  map.
  \begin{enumerate}
  \item Prove that there is a unique linear one-to-one and onto map
    $\hat A : X_1/\ker A \to X_2$ such that $A = \hat A \circ \pi$.

  \item If in addition, $X_1$ is a \nls\ and $\ker A$ is closed, then
    using the quotient norm $ \| \pi(u) \|_{X_1/\ker A} = \inf_{w \in
      \ker A } \| u + w \|_{X_1}$, prove that 
    \begin{equation}
      \label{eq:12}
      \| y \|_{X_2} = \|
      \hat A^{-1} y \|_{X_1/\ker A}   
    \end{equation}
    makes $X_2$ into a \nls\ and $\hat A$ establishes an isometric
    isomorphism between $X_1/\ker A$ and~$X_2$.
  \end{enumerate}
\end{exercise}


\begin{exercise}
  For all $r \in \Hdiv\om$, define $\trn r \in H^{-1/2}(\d\oh)$ by
  $\trn r|_{\d K} = r\cdot n |_{\d K}$.  What is $\ker(\trn)$? Verify
  that $\ker(\trn)$ is a closed subspace of $\Hdiv\om$.  Apply
  Exercise~\ref{x:quotient} with $X_1=\Hdiv\om$, $X_2 =
  H^{-1/2}(\d\oh)$, and $A =\trn$, to conclude that the norm
  in~\eqref{eq:13} is the same as~\eqref{eq:12}, and that
  $H^{-1/2}(\d\oh)$ is complete under that norm.
\end{exercise}

\begin{exercise}
  \label{x:Ext}
  Prove that there is a \cts\ linear map $E : H^{-1/2}(\d\oh) \to
  \Hdiv\om$ such that $\| \hr_n \|_{H^{-1/2}(\d\oh)} = \| E \hr_n
  \|_{\Hdiv\om}$ and $\trn E \hr_n = \hr_n$. ({\em Hint:} Consider
  $\hat A^{-1}$ from Exercise~\ref{x:quotient} and find a minimizer
  over a coset.)
\end{exercise}

We now set 
\[
b( \,(w,\hr_n), \, v) = (\grad w, \grad v)_\oh - \ip{
  \hr_n, v}_{\d\oh},\qquad
\ell(v) = (f,v)_\om 
\]
and proceed to analyze the formulation~\eqref{eq:Laplacian}. Let
$\Hdiv\oh = \{ v \in L^2(\om)^N: v|_K \in \Hdiv K $ for all $ K\in\oh
\}$. Define
  \begin{align}
    \label{eq:jmpflux}
    \jmpn{ \tau\cdot n}
    & \defn
    \sup_{ 0\ne \phi \in H_0^1(\om)}  
    \frac{ |\ip{\tau  \cdot n, \phi }_{\d\oh}| }
    { \; \| \phi \|_{H^{1}(\om)} }, 
    && \forall \tau \in \Hdiv\oh, 
    \\ 
    \label{eq:jmptrace}
    \jmpn{ v n }
    & \defn
    \sup_{ 0\ne r \,\in \Hdiv\om} 
    \frac{ |\ip{ r \cdot n,v }_{\d\oh}| }
    { \; \| r \|_{\Hdiv\om} },
    && \forall v \in H^1(\oh).
  \end{align}
\begin{exercise}
  \label{x:jumpcts}
  Prove that any $v \in H^1(\oh)$ has $\jmpn{ vn} =0$ if and only if
  $v\in H_0^1(\om)$.
\end{exercise}


\begin{exercise}
  \label{x:jump1}
  Prove that 
  $\displaystyle{
  \jmpn{ v n }
   =
    \sup_{0\ne \hr_n \in H^{-1/2}(\d\oh)} 
    \frac{ \!\!| \ip{\hr_n,v}_{\d\oh} | }
    { \quad \| \hr_n \|_{H^{-1/2}(\d\oh)} } 
    .}$
\end{exercise}

Next, define an orthogonal projection $P:L^2(\om)^N \to \grad
H_0^1(\om)$, by
\begin{equation}
  \label{eq:P}
  ( P q, \grad \phi)_\om = (q,\grad \phi)_\om \qquad \forall \phi \in H_0^1(\om).
\end{equation}
\begin{exercise}
  Prove that $\grad H_0^1(\om)$ is a closed subspace of $L^2(\om)^N$
  (under the current assumptions on $\om$).
\end{exercise}

\begin{lemma}[A \Poincare-type inequality]
  \label{lem:Poincare}
  There is a positive constant $C$ independent of $\oh$ such
  that for all $v$ in $H^1(\oh)$,
  \[
  C \| v \|_\oh \le 
    \| P\grad v \|_\oh + \jmpn{ v  n }.
  \]
\end{lemma} 
\begin{proof}
  Let $\phi$ in $H_0^1(\om)$ solve the Dirichlet problem $-\Delta \phi
  = v$.  Then,
  \begin{align*}
    \| v \|_{{L^2(\om)}}^2 
    & = (-\Delta \phi,v)_\om
    = (\grad \phi, \grad v)_\oh - \ip{  \frac{\d \phi}{\d n },v }_{\d\oh}
    \\
    &= (\grad \phi, P\grad v)_\oh - \ip{  \frac{\d \phi}{\d n },v }_{\d\oh}
      \\
    & \le \| P\grad v \|_\oh \| \grad \phi \|_\oh + 
    \bigg(
    \frac{ \ip{\grad \phi \cdot n,v }_{\d\oh} } { \quad\| \grad \phi \|_{\Hdiv\om} }
    \bigg)  \| \grad \phi \|_{\Hdiv\om}
    \\
    & \le \left(
      \|P \grad v \|_\oh+
      \sup_{  q  \in \Hdiv\om } \!\!
      \frac{ \ip{v,  q  \cdot n }_{\d\oh} } { \quad\|  q  \|_{\Hdiv\om} }
      \right)
        \| \grad \phi \|_{\Hdiv\om}.
  \end{align*}
  Since $\om$ is bounded and connected, the standard \Poincare\
  inequality holds, so $\| \grad \phi \|^2_{{L^2(\om)}} \le C \| v
  \|^2_{L^2(\om)}$. Moreover, $\|\dive( \grad \phi)\|_{L^2(\om)} = \|
  v \|_{L^2(\om)}$, so the result follows.
\end{proof}

\begin{lemma}[Piecewise harmonic functions]
 \label{lem:harmonic} 
 There is a $C>0$ independent of $\oh$ such that for all 
 $v$ in $H^1(\oh)$ satisfying $\Delta (v|_K) =0$ for all $K\in
 \oh$, we have 
 \[ 
 C \| \grad v \|_\oh
   \le
   \jmpn{\grad v\cdot n }
   + \jmpn{v  n}.
 \]
\end{lemma}
\begin{proof}
  Let $\tau=\grad v$.  We construct the Helmholtz-Hodge decomposition
  of $\tau$, namely $\tau = \grad \psi + z$ with $\psi$ in
  $H_0^1(\om)$ and $z$ in $\Hdiv\om$, as follows: First define $\psi$
  by
  \begin{equation}
    \label{eq:prev15}
    (\grad \psi, \grad \varphi)_\om = (\tau,\grad \varphi)_\oh,
    \qquad \forall \varphi \in H_0^1(\om).
  \end{equation}
  Then, set $z = \tau - \grad \psi$. 

  By~\eqref{eq:prev15}, $ (\tau - \grad \psi, \grad \varphi)_\om = 0$,
  so $\dive z=0$. Hence the two components, $ \grad \psi$ and $z,$ are
  $L^2(\om)$-orthogonal and
  \begin{equation}
    \label{eq:16'}
     \| z  \|_\om^2
     + 
     \| \grad \psi \|^2_\om 
     =\| \tau \|_{\om}^2.
   \end{equation}
   Thus,
   \begin{align*}
     \| \tau \|_\om^2
     & =  ( \tau,\tau) = (  \tau,   \grad \psi + z )_\oh
      = (  \tau,   \grad \psi)_\oh + (\grad v,  z )_\oh
      \\
     & = -(\dive \tau, \psi)_\oh + \ip{\tau\cdot n,\psi }_{\d\oh}
        + \ip{ n \cdot z,v}_{\d\oh},
   \end{align*}
   Since $v$ is harmonic on each element, the first term
   vanishes. Hence
  \begin{align*}
      \| \tau \|_{\om}^2 
      &
      =     
      \frac{\ip{\tau\cdot n,\psi }_{\d\oh}}{ \| \psi \|_{H^1(\om)} } 
      \| \psi \|_{H^1(\om)}  
      +
      \frac{\ip{  n \cdot z,v }_{\d\oh} }
      {\|  z \|_{\Hdiv\om} }  \|  z \|_{L^2(\om)}
      \\
      & \le
      \bigg(\sup_{w\in H_0^1(\om)}
       \frac{\ip{ \tau\cdot n,w }_{\d\oh}}{ \| w \|_{H^1(\om)} } \bigg)
       \| \psi \|_{H^1(\om)}  
      + 
      \bigg(
      \sup_{ q  \in \Hdiv\om} 
      \frac{\!\ip{  n \cdot q,v }_{\d\oh} }
      {\quad\|  q \|_{\Hdiv\om} } \;
      \bigg) 
      \|  z \|_{L^2(\om)}.
    \end{align*}
    The result now follows from~\eqref{eq:16'} and the standard
    \Poincare\ inequality applied to $\psi$. 
\end{proof}

\begin{lemma}
  \label{lem:Pharmonic}
  There is a positive constant $C$ independent of $\oh$ such
  that for all $v$ in $H^1(\oh)$,
  \[
  \| \grad v \|_\oh \le  \| P \grad v \|_\oh + 
  C \jmpn{ v n }.
  \]
\end{lemma}
\begin{proof}
  Let $z \in H_0^1(\om)$ be such that $P\grad v = \grad z$, let
  $\veps = v-z$, and let $r|_K = -\grad( \veps|_K)$ on all $K \in
  \oh$. Then,~\eqref{eq:P} implies
  \begin{equation}
    \label{eq:prev2}
    (P\grad v-\grad v,\grad \phi)=
    (r, \grad \phi)_K =0 \qquad \forall \phi \in H_0^1(\om).
  \end{equation}
  Choosing $\phi \in \DD(K)$, we immediately find that $\dive (r|_K)
  =0$, i.e., $\veps$ is harmonic on each $K\in \oh$.  Applying
  Lemma~\ref{lem:harmonic}, we thus obtain
  \begin{equation}
    \label{eq:3}
    C\| \grad \veps\|_\oh \le 
      \jmpn{ r \cdot n } + \jmpn{ \veps n}.
  \end{equation}
  But $\jmpn{r \cdot n} =0$. This is because we may integrate by parts
  element by element to conclude from~\eqref{eq:prev2} that
  \[
  0=(r,\grad\phi)_\oh = -(\dive r, \phi)_\oh + \ip{ r\cdot n, \phi}_{\d\oh} 
  =  \ip{ r\cdot n, \phi}_{\d\oh},
  \]
  for all $\phi \in H_0^1(\om)$, so definition~\eqref{eq:jmpflux}
  implies $\jmpn{r \cdot n} =0$. Moreover,
  definition~\eqref{eq:jmptrace} shows that $\jmpn{ \veps n } = \jmpn{
    v n }$. Therefore, returning to~\eqref{eq:3}, we conclude that
  \begin{align*}
    \| \grad v \|_\oh 
    \le \| P\grad v  \|_\oh + \|\grad \veps \|_\oh
    \le 
     \| P\grad v  \|_\oh + C \jmpn{v n},
  \end{align*}
  which proves the lemma.
\end{proof}

\begin{theorem}
  \label{thm:laplacian}
  Assumption~\ref{asm:wellposed} holds for the
  formulation~\eqref{eq:Laplacian}.
\end{theorem}
\begin{proof}
  The uniqueness part of Assumption~\ref{asm:wellposed}, namely $\{
  (w,\hs_n) \in X: \;b(\,(w,\hs_n), y)=0,\; \forall y\in Y\} = \{0\}$
  can be proved by an argument analogous to what we have seen
  previously (see between~\eqref{eq:9} and~\eqref{eq:10}), so is left
  as an exercise. 

  To prove the continuity estimate, we use $|\ip{ \hs_n, v}_{\d\oh}
  |\le \| \hs_h \|_{H^{-1/2}(\d\oh)} \jmpn{ v n}$, a consequence of
  Exercise~\ref{x:jump1}, to get 
  \[
  |b (\,(w,\hs_n), y)| 
  \le \| (w,\hs_n) \|_X \left( \| \grad v
    \|_\oh^2 + \jmpn{ vn } \right)^{1/2}
  \]
  Now, since~\eqref{eq:jmptrace} implies
    \begin{align*}
    \jmpn{ v n }
    & 
  = \sup_{ r \,\in \Hdiv\om} 
    \frac{ (r, \grad v)_\oh + (\dive r, v)_\oh }
    { \; \| r \|_{\Hdiv\om} }
    \le \| v \|_Y,
  \end{align*}
  the continuity estimate is proved.

  It only remains to prove the inf-sup condition. But 
  \begin{align*}
    \sup_{\,(w,\hs_n)\in {{X}}}
    \frac{ |b(\,(w,\hs_n),v)| }{ \| (w,\hs_n) \|_{{X}}}
    &
    \ge  \sup_{w\in H_0^1(\om) } \frac{ (\grad w, \grad
    v)_\oh }{ \| \grad w \|_{L^2(\om)} }
  = \| P \grad v \|_{L^2(\om)},
  \\
  \sup_{\,(w,\hs_n)\in {{X}}}
    \frac{ |b(\,(w,\hs_n),v)| }{ \| (w,\hs_n) \|_{{X}}}
    &
    \ge
  \sup_{\hs_n \in H^{-1/2}(\d\oh)}
  \frac{ |\ip{\hs_n, v}_{\d\oh} |}
       { \| \hs_n \|_{H^{-1/2}(\d\oh)}  }
       = \jmpn{ vn},
  \end{align*}
  so the required inf-sup condition follows by adding and using 
  Lemmas~\ref{lem:Pharmonic} and~\ref{lem:Poincare}.
\end{proof}

To consider a particular instance of the DPG method, we now fix element
shapes to be triangles.  For any integer $p\ge 0$ let $P_p(K)$ denote
the space of polynomials of degree at most $p$ restricted to $K$.  For
any triangle $K$, let $P_p(\d K)$ denote the set of functions on $\d
K$ whose restrictions to each edge of $K$ is a polynomial of degree at
most $p$. We now set
\begin{subequations}
  \label{eq:discretespaceslaplace}
  \begin{align}
    \label{eq:Xhlap}
    X_h
    & = \{ (w, \hs_n) \in X: \; w|_K \in P_{p+1}(K), \; 
     \hs_n|_{\d K} \in P_p(\d K) \; \;\forall K\in \oh \},
     \\
     \label{eq:Yrlap}
    Y^r & = \{ v\in Y: \; v|_K \in P_r(K),\; \;\forall K\in \oh \},
  \end{align}
\end{subequations} 
compute the inexact test space $Y_h^r$, and consider
the DPG method that finds $(u_h, \hqh)\in X_h$ solving 
\begin{equation}
  \label{eq:pdpglaplacian}
  (\grad u_h, \grad v)_\oh - \ip{
  \hqh, v}_{\d\oh}
  = (f,v)_\om,\qquad\forall v\in Y_h^r. 
\end{equation}

\begin{theorem}
  \label{thm:laplacian2}
  Suppose $N=2$, $\oh$ is a shape regular finite element mesh of
  triangles and $X_h$ and $Y^r$ are set as
  in~\eqref{eq:discretespaceslaplace}.  Then, whenever $r \ge p+N$,
  Assumption~\ref{asm:Pi} holds. Consequently, by
  Theorem~\ref{thm:inex}, the DPG method~\eqref{eq:pdpglaplacian} is
  quasioptimal.
\end{theorem}
\begin{proof}
  Let $r=p+N$. It is easy to see that for every $v\in H^{1}(K)$, there
  is a unique $\vpi_r^0v\in P_r(K)$ satisfying
\begin{align*}
  \vpi_r^0 v & =0 
  &&\text{at all 3 vertices of } K,
  \\
  (\vpi_r^0 v - v,q_{p-1})_{K}
  & =0,
  && \forall q_{p-1}\in P_{p-1}(K),\\
  \langle \vpi_r^0 v - v, \mu_{p}\rangle_{\partial K}
  & =0,
  && \forall \mu_{p}\in P_{p}(\partial K).
\end{align*}
Setting $\vpi v = \vpi_r^0 (v-\bar v) + \bar v$, where $\bar v $
denotes the mean value of $v$ on $K$, it is an exercise to show that
there is a $C$ independent of the size of the triangle $K$ (but
dependent on the shape regularity of $K$) such that
\begin{subequations}
\begin{align}
  \label{eq:20}
  (\vpi v - v,q_{p-1})_{K}
  & =0,
  && \forall q_{p-1}\in P_{p-1}(K),\\
  \label{eq:25}
  \langle \vpi v - v, \mu_{p}\rangle_{\partial K}
  & =0,
  && \forall \mu_{p}\in P_{p}(\partial K),
  \\
  \| \vpi v \|_{H^1(K) } 
  & \le C \| v\|_{H^1(K)}
  && \forall v \in H^1(K).
\end{align}
\end{subequations}
Then,
\begin{align*}
  b(\, (w_h, \hsh), v - \vpi v) 
  & =
  (\grad w_h,  \grad (v - \varPi v) )_\oh - \ip{ \hsh, v-\varPi v}_{\d\oh} 
  \\
  & =
  -(\Delta w_h,  v - \varPi v )_\oh - 
  \ip{ \hsh - n\cdot \grad w_h, v-\varPi v}_{\d\oh} =0,
\end{align*}
by~\eqref{eq:20} and~\eqref{eq:25}.
\end{proof}

\begin{example}  \label{eg:laplacehybrid}
To put this method into the framework of
Assumption~\ref{asm:hybrid}, set 
\begin{gather*}
X_0 = H_0^1(\om),\quad
\hX = H^{-1/2}(\d\oh), 
\\
b_0(u,y) = (\grad u, \grad y)_\oh,\quad
\hb(\hq_n,y) = -\ip{\hq_n,y}_{\d\oh}.
\end{gather*}
Furthermore, the $X_h$ in~\eqref{eq:Xhlap} can be split into $\Xoh
\times \hX_h$  with 
\begin{align*}
  \Xoh & = \{ w\in H_0^1(\om): \; w|_K \in P_{p+1}(K),\;\;\forall K\in \oh \},
  \\
  \hX_h & = \{  \hs_n \in  H^{-1/2}(\d\oh): \;
             \hs_n|_{\d K} \in P_p(\d K) \; \;\forall K\in \oh \},
\end{align*}
so  that $Y_0$ in~\eqref{eq:21} becomes 
\[
\Yo= \{ y\in H^1(\oh) : \ip{ \hsh,y}_{\d\oh}=0,\; \forall \hsh
\in \hX_h\}, 
\]
a {\bf {\em{weakly conforming}}} subspace of $H^1(\om)$. Its subspace,
defined in~\eqref{eq:30} becomes $ Y_0^r = \{ y\in Y^r : \ip{
  \hsh,y}_{\d\oh}=0,\; \forall \hsh \in \hat X_h\}$. The
non-hybrid form of the DPG method, namely~\eqref{eq:27} uses this
$Y_0^r$ and finds $u_h\in \Xoh$ satisfying
\begin{equation}
  \label{eq:14}
  (\grad u_h, \grad y_0)_\oh = (f,y_0) \qquad \forall y_0\in \Yohr.
\end{equation}
Recall that $y_0\in \Yohr$ if and only if it is in $Y_0^r$ and solves 
\begin{equation}
  \label{eq:28}
  (\grad y_0, \grad v)_\oh + (y_0,v)_\oh = (\grad w,\grad v)_\oh
  \qquad \forall v \in Y^r
\end{equation}
for some $w \in \Xoh$.  By Theorem~\ref{thm:pdpghybrid}, the $u_h$
in~\eqref{eq:14} coincides with the first solution component of the
hybrid DPG method~\eqref{eq:pdpglaplacian}. The difficulty with
implementing~\eqref{eq:14} is that the computation of $\Yohr$,
requiring multiple solves of the global weakly conforming
problem~\eqref{eq:28}, is too expensive. In contrast the hybrid
form~\eqref{eq:pdpglaplacian} is easily implementable as the
computation of $Y_h^r$ amounts to inverting a block diagonal matrix.
\end{example}

Before concluding, let us consider {\bf{\em{convergence rates}}}. The
error estimate of Theorem~\ref{thm:inex} (which holds by virtue of
Theorems~\ref{thm:laplacian} and~\ref{thm:laplacian2}) gives
\[
\| u - u_h \|_{H^1(\om)} + \| \hq_n - \hqh \|_{H^{-1/2}(\oh)} 
\le C \inf_{ (w_h,\hsh)\in X_h} 
\left(
\| u - w_h \|_{H^1(\om)} + \| \hq_n - \hsh \|_{H^{-1/2}(\oh)} 
\right).
\]
Henceforth $C>0$ denotes a generic constant independent of $
h=\max_{K\in \oh} \diam(K) $ but dependent on the mesh's shape
regularity.  To obtain convergence rates in terms of $h$, we must
bound the infimum above. Suppose $u$ is smooth.  By the
Bramble-Hilbert lemma,
\begin{equation}
  \label{eq:31}
  \inf_{w_h\in \Xoh}  \| u - w_h \|_{H^1(\om)} \le C h^{p+1} | u |_{H^{p+2}(\om)}.  
\end{equation}
For the error in $\hq_n$, let $q =\grad u$ and let $\pirt q$ denote
the Raviart-Thomas projection of $q$ into $\{ r\in \Hdiv\om: r|_K \in
P_p(K) + x P_p(K), \;\forall K\in\oh\}$. Then $\pirt q \cdot n \in
\hX_h$, so
\[
\inf_{\hsh\in \hX_h}  
\| \hq_n - \hsh \|_{H^{-1/2}(\oh)}
\le
\| (q - \pirt q) \cdot n \|_{H^{-1/2}(\oh)}
\le
\| q - \pirt q \|_{\Hdiv\om}
\]
where we have used~\eqref{eq:13}, by which, the $H^{-1/2}(\oh)$-norm
of a function can be bounded by the $\Hdiv\om$-norm of any of its
extensions. Estimating $\| q - \pirt q \|_{\Hdiv\om} $ as usual,
\begin{equation}
  \label{eq:32}
  \inf_{\hsh\in \hX_h}  
  \| \hq_n - \hsh \|_{H^{-1/2}(\oh)}
  \le C h^{p+1}\left( | q|_{H^{p+1}(\om)} + |\dive q|_{H^{p+1}(\om)} \right).
\end{equation}
From~\eqref{eq:31} and~\eqref{eq:32}, we obtain $O(h^{p+1})$
convergence for $u_h$ and $\hqh$.

\begin{wraptable}
  [7] 
  {r}
  [.0cm]  
  {2.2in} 
  \centering{\footnotesize{
   \begin{tabular}{|l|cc|}
    \hline
    $h/\sqrt2$    &$\|u-u_h\|_{H^1(\om)}$&$\|\veps^r\|_{H^1(\oh)}$
    \\     \hline
    1/4 &0.008277&0.008987
    \\
    1/8 &0.002111&0.002297
    \\
    1/16&0.000531&0.000579
    \\
    1/32&0.000133&0.000145
    \\
    1/64&0.000033&0.000036
    \\ \hline
  \end{tabular}
}}
\end{wraptable}
Let us now check if we see this convergence rate in practice.  We use
a \fenics\ code (download code from
\href{http://web.pdx.edu/~gjay/pub/dpglaplace.py}{\underline{here}})
which implements the mixed reformulation of the DPG method given in
Theorem~\ref{thm:pdggmixed} (see also Remark~\ref{rem:impl}). Solving
a simple problem with a smooth solution (see the code for details) on
the unit square, using $p=1$ and uniform meshes with various $h$, we
collect the results in the table aside.  Clearly, $\| u - u_h
\|_{H^1(\om)}$ appears to converge at $O(h^2)$, in accordance with the
theory. Also, the error estimator $\veps^r$ (see
Definition~\ref{def:estimator}) appears to converge to zero at the
same rate as the error.

\begin{figure}
  \centering      
  \includegraphics[trim=5cm 4cm 5cm 4cm, clip, width=0.3\textwidth]{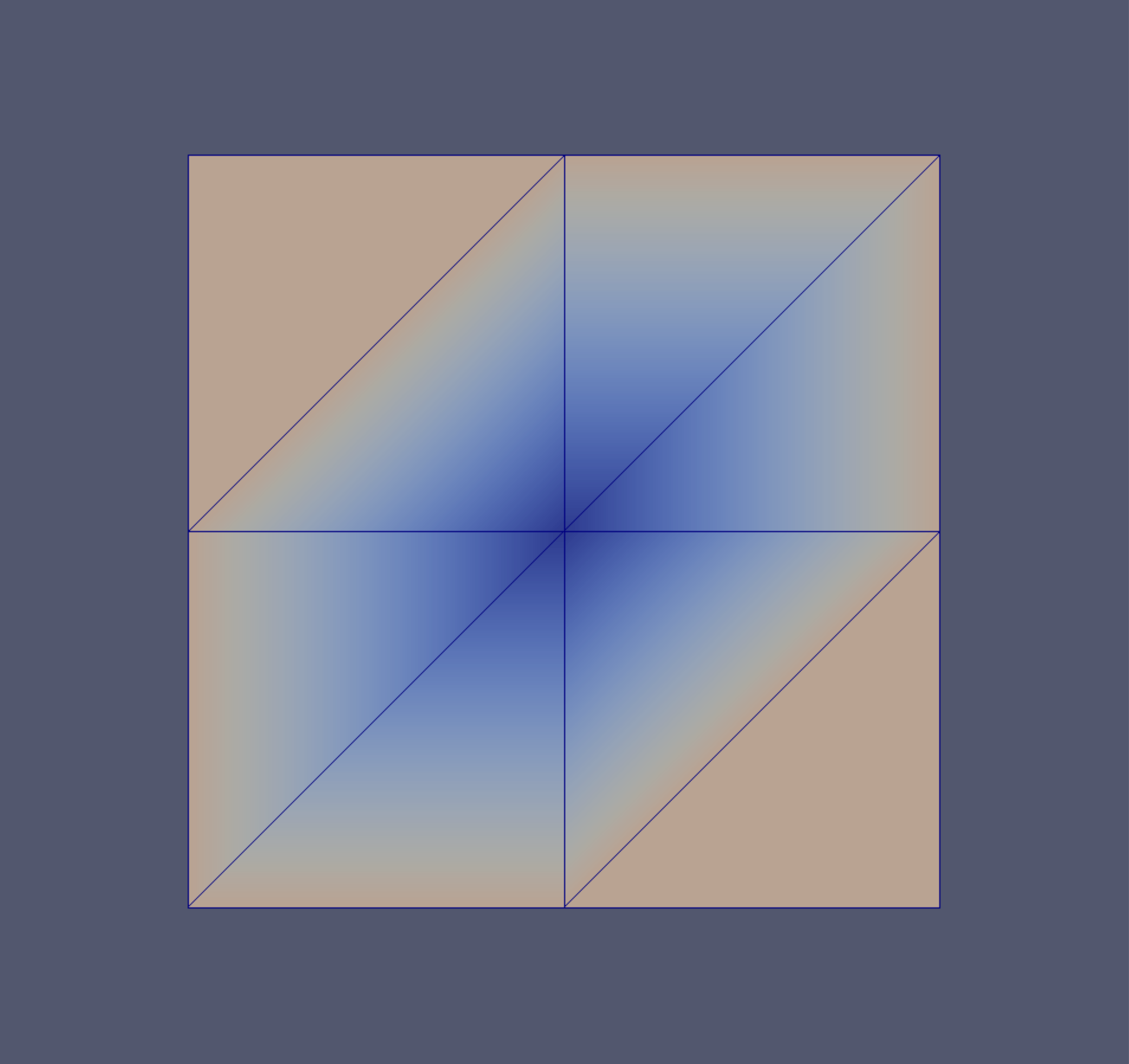}
  \includegraphics[trim=5cm 4cm 5cm 4cm, clip, width=0.3\textwidth]{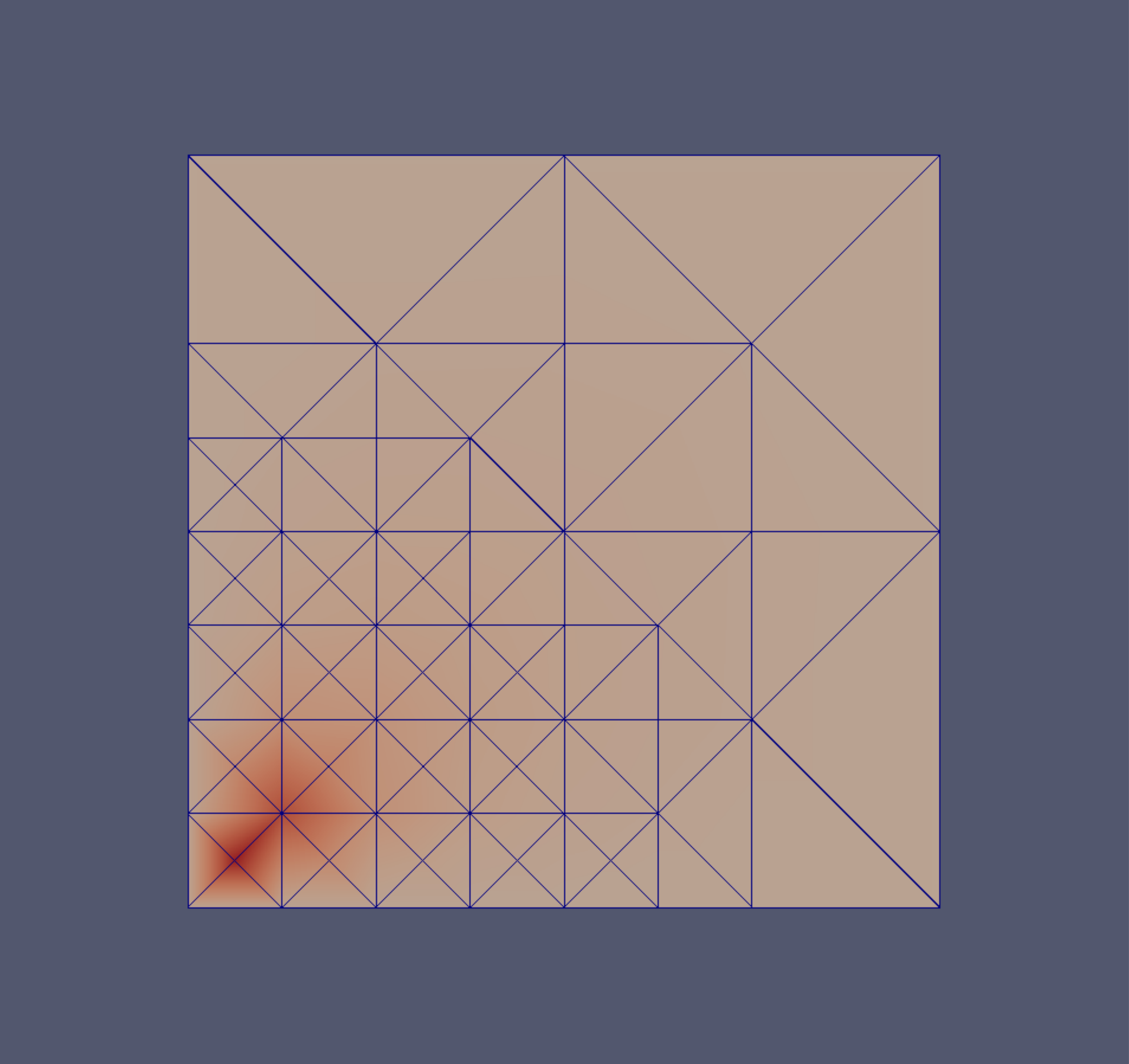}
  \includegraphics[trim=5cm 4cm 5cm 4cm, clip, width=0.3\textwidth]{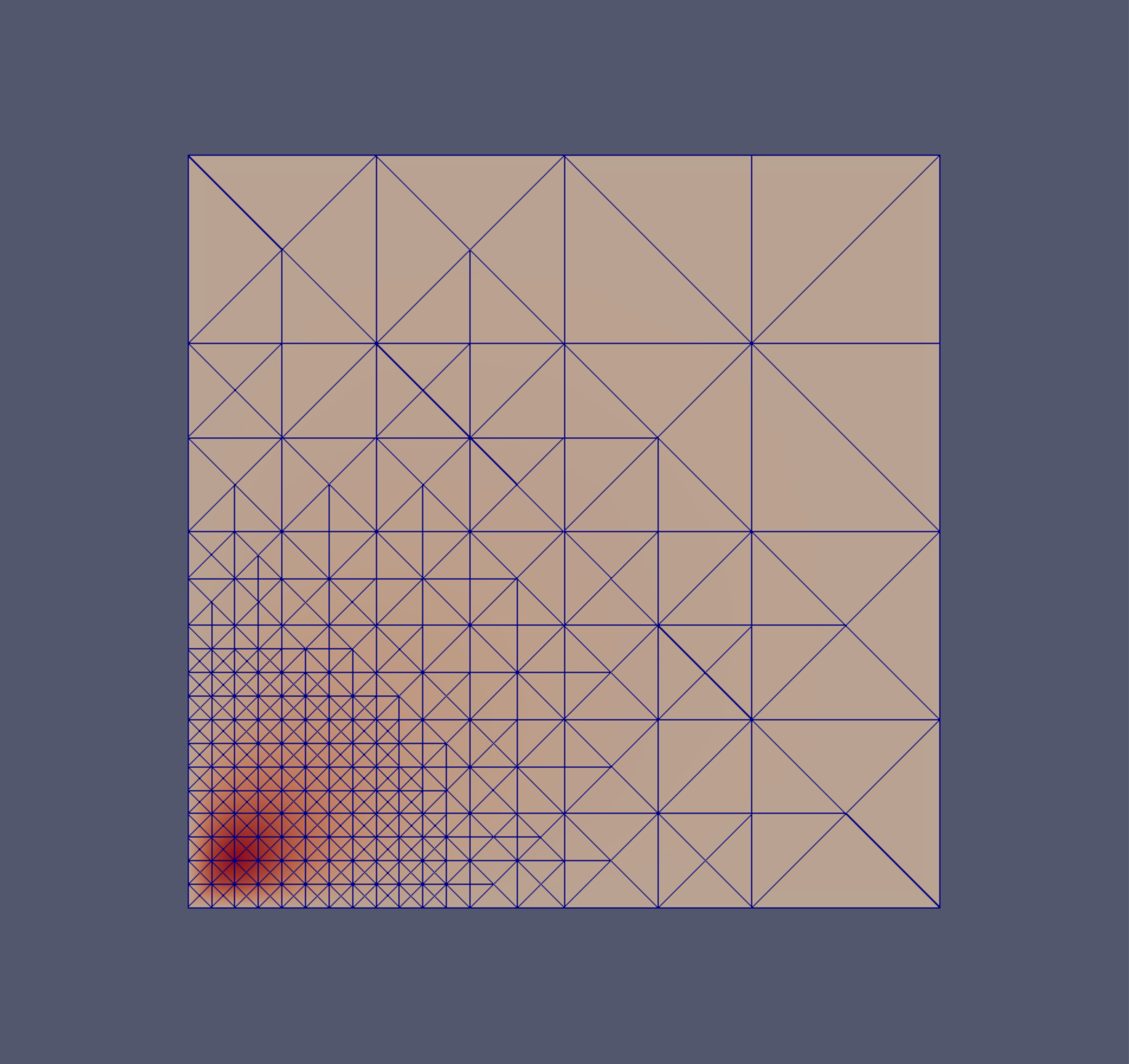}
  \caption{Initial, midway, and final iterates in an adaptive scheme
    using the DPG error estimator~$\veps^r$.}
  \label{fig:adapt}
\end{figure}

It is possible to prove that the error estimator $\veps^r$ is an
efficient and reliable indicator of the actual error, but to keep
these lectures introductory, we omit the details. Instead, let us consider 
a \fenics\ implementation of a typical adaptive algorithm using the
element-wise norms of $\veps^r$ as the error indicators (download the
code from
\href{http://web.pdx.edu/~gjay/pub/dpglapladapt.py}{\underline{here}}). In
the code, we compute the element error indicator
$\|\veps^r\|_{H^1(K)}$ on each $K\in\oh$, and sort the elements in
decreasing order of the indicators. The elements falling in the top
half are marked for refinement. In the next iteration of the adaptive
algorithm, those elements (and possibly other adjacent elements) are
refined by bisection, the DPG problem is solved on the new mesh, and
the newly obtained $\veps^r$ is used to mark elements as before. We
use this process to approximate the solution of the Dirichlet
problem~\eqref{eq:bvp} on the unit square  with
$
f = e^{-100(x_0^2+x_1^2)}.
$
We expect the solution to have interesting variations only near the
origin. As seen in Figure~\ref{fig:adapt}, the error estimator
automatically identifies the right region for refinement even though
we started with a very coarse mesh.

\newpage

\section*{Appendix}

\bigskip

{\footnotesize{

\subsection*{Codes}

The programs are in the python \fenics\ environment. You will need to
download and install \fenics\ from \href{http://fenicsproject.org}
{\underline{\smash[b]{{fenicsproject.org}}}} for them to run. (I am not
an expert in \fenics\ and suggestions to improve the codes are very
welcome.) Here are the available downloads on DPG methods:

\begin{itemize}
\item The \fenics\ code for implementing the Petrov Galerkin method of
  Example~\ref{eg:odeWIP} and generating Figure~\ref{fig:spectralPG}
  can be downloaded from
  \href{http://web.pdx.edu/~gjay/pub/1Dpgode.py}{\underline{here}}. The
  code also implements a comparable least square method and the
  computation of $L^2$ projections.
\item You can
  \href{http://web.pdx.edu/~gjay/pub/dpglaplace.py}{\underline{download}}
  download a \fenics\ implementation of the DPG method for the
  Dirichlet problem.
\item A code implementing an adaptive algorithm using the DPG error
  estimator is also
  \href{http://web.pdx.edu/~gjay/pub/dpglapladapt.py}{\underline{available}}. This
  code is modeled after a \fenics\ demo for
  standard finite elements.

\end{itemize}

\subsection*{Acknowledgments}

These are (unpublished) notes from a few of my lectures in a Spring
2013 graduate class at PSU (MTH 610). The DPG research is in close
collaboration with Leszek Demkowicz (see references below). I am
grateful to my students for their feedback on the notes and to
Kristian \O lgaard for clarifying \fenics\ syntax. I am also grateful
to NSF and AFOSR for supporting my research into DG and mixed methods
and for encouraging the integration of such research into graduate
education.

\subsection*{Bibliographic remarks}

The presentation in Section~\ref{sec:optimal-test-spaces}, including
the terminology of `optimal test spaces', Theorem~\ref{thm:quasiopt},
etc.~is based on~\cite{DemkoGopal11}. The DPG methods were developed
in a series of papers, beginning
with~\cite{DemkoGopal10a,DemkoGopal11}. The name ``DPG'' was
previously used by others~\cite{BottaMicheSacco02}, but without the
concept of optimal test functions. The interpretation as a mixed
formulation (Theorem~\ref{thm:mixed}) is motivated
by~\cite{DahmeHuangSchwa11a}.  Theorem~\ref{thm:inex} is
from~\cite{GopalQiu12a}. Operators such as $\varPi$, in the standard
mixed Galerkin context, are sometimes known as Fortin operators.
Theorems~\ref{thm:hybrid} and~\ref{thm:pdpghybrid} have not appeared
in this form previously.  Lemmas~\ref{lem:Poincare}
and~\ref{lem:harmonic} are from~\cite{DemkoGopal11a}, but the method
of Section~\ref{sec:laplacian} was developed later, independently
in~\cite{DemkoGopal13} and~\cite{BroerSteve12}.  A more comprehensive
bibliography is available in~\cite{DemkoGopal13rev}.

}}
\end{document}

%% file: preamble.tex
\def\d{\partial}

\newcommand{\Babuska}{Babu{\v{s}}ka}       
\newcommand{\Poincare}{Poincar{\'{e}}}     

\newcommand{\RRR}{\mathbb{R}}

\newcommand{\CCC}{\mathbb{C}}

\newcommand{\dive}{\mathop\mathrm{div}}
\newcommand{\grad}{\ensuremath{\mathop{{\bf{grad}}}}}

\newcommand{\diam}{\ensuremath{\mathop\mathrm{diam}}}

\newcommand{\defn}{\;\stackrel{\mathrm{def}}{=}\;}

\newcommand{\ran}{\mathop\mathrm{ran}}

\newcommand{\supp}{\mathop\mathrm{supp}}

\newcommand{\Hdiv}[1]{\bH(\dive,#1)}

\newcommand{\DD}{\mathcal{D}}

\newcommand{\veps}{\varepsilon}

\newcommand{\vpi}{\varPi}

\newcommand{\ntrip}[1]{\vvvert {#1} \vvvert}  
\newcommand{\ip}[1]{\langle {#1} \rangle}

\newcommand{\matlab}{MATLAB\textregistered\renewcommand{\matlab}{MATLAB}}
\newcommand{\femlab}{FEMLAB\textregistered\renewcommand{\femlab}{FEMLAB}}

\newcommand{\tfae}{the following are equivalent}
\newcommand{\Tfae}{The following are equivalent}
\newcommand{\cts}{continuous}

\newcommand{\nls}{normed linear space}
\newcommand{\nlss}{normed linear spaces}
